\documentclass[11pt,a4paper]{amsart}
\usepackage{newtx}
\usepackage{hyperref}
\usepackage{graphicx}
\usepackage{stmaryrd}
\usepackage{tikz,tkz-euclide}
\usepackage{amsmath}
\usepackage{mathrsfs}
\usepackage{mathdots}
\usepackage[inline]{enumitem}

\usepackage[utf8]{inputenc}
\usepackage[T1]{fontenc}
\usepackage[left=3cm, right=3cm]{geometry}

\newtheorem{theorem}{Theorem}[section]

\newtheorem{lemma}[theorem]{Lemma}

\theoremstyle{definition}
\newtheorem{definition}[theorem]{Definition}
\newtheorem{remark}[theorem]{Remark}
\newtheorem{example}[theorem]{Example}

\numberwithin{equation}{section}
\numberwithin{figure}{section}

\newcommand{\C}{\mathbb{C}}
\newcommand{\R}{\mathbb{R}}
\newcommand{\Z}{\mathbb{Z}}

\newcommand{\GL}{\mathbf{GL}}
\newcommand{\U}{\mathbf{U}}
\newcommand{\SU}{\mathbf{SU}}
\newcommand{\Sp}{\mathbf{Sp}}
\newcommand{\SL}{\mathbf{SL}}
\newcommand{\SO}{\mathbf{SO}}
\newcommand{\Rad}{\mathbf{R}}
\newcommand{\Zbf}{\mathbf{Z}}
\newcommand{\Xbf}{\mathbf{X}}
\newcommand{\zbf}{\mathbf{z}}
\renewcommand{\i}{\mathbf{i}}
\renewcommand{\O}{\mathbf{O}}

\newcommand{\cG}{\mathcal{G}}
\newcommand{\cP}{\mathcal{P}}

\newcommand{\Hom}{\mathrm{Hom}}
\newcommand{\Span}{\mathrm{span}}
\newcommand{\Lie}{\mathrm{Lie}}
\newcommand{\Fr}{\mathrm{Fr}}
\newcommand{\HN}{\mathrm{HN}}
\newcommand{\Iso}{\mathrm{Iso}}
\newcommand{\SIso}{\mathrm{SIso}}
\newcommand{\Symp}{\mathrm{Symp}}
\newcommand{\SOrth}{\mathrm{SOrth}}
\newcommand{\rk}{\mathrm{rk}}
\newcommand{\Ad}{\mathrm{Ad}}
\newcommand{\ad}{\mathrm{ad}}
\newcommand{\ab}{\mathrm{ab}}
\newcommand{\Mat}{\mathrm{Mat}}
\newcommand{\tr}{\mathrm{tr}}
\newcommand{\Diag}{\mathrm{diag}}
\newcommand{\conj}{\mathrm{conj}}

\newcommand{\sst}{\mathrm{sst}}

\renewcommand{\k}{\underline{k}}
\renewcommand{\u}{\underline{u}}
\renewcommand{\a}{\underline{a}}
\renewcommand{\b}{\underline{b}}

\renewcommand{\leq}{\leqslant}
\renewcommand{\geq}{\geqslant}

\begin{document}

\title{Harder-Narasimhan filtrations of decorated vector bundles}

\author{Emanuel Roth}
\address{School of Mathematics, James Clerk Maxwell Building, Peter Guthrie Tait Road, Edinburgh EH9 3FD, United Kingdom}
\email{eroth2@ed.ac.uk}

\author{Florent Schaffhauser}
\address{Mathematisches Institut, Universität Heidelberg, Im Neuenheimer Feld 205,
    69120 Heidelberg, Germany}
\email{fschaffhauser@mathi.uni-heidelberg.de}


\begin{abstract}
    A \textit{decorated vector bundle} is a vector bundle equipped with a reduction of structure group to a complex reductive subgroup $G \subseteq \mathbf{GL}(r,\mathbb{C})$. Examples include symplectic and special-orthogonal vector bundles, as well as vector bundles with trivial determinants. In this expository paper, we provide direct constructions of Harder-Narasimhan filtrations of symplectic and special-orthogonal vector bundles, and use them to construct canonical reductions in the sense of Atiyah and Bott. We compare these canonical reductions to those constructed by Biswas and Holla. Lastly, we set up the obstruction theory necessary to define Harder-Narasimhan types of principal bundles, and stratify the moduli stack of principal $G$-bundles.
\end{abstract}

\maketitle
\tableofcontents
\section{Introduction}\label{sec_intro}

Geometric invariant theory (GIT) was developed by David Mumford as a way to construct quotients and solve moduli problems in algebraic geometry \cite{mumf_projectiveinvar}. The basic notion in this theory is that of a categorical quotient: given a $G$-action on a variety $X$, the goal is to find sufficient conditions on $G$ and $X$ so that an arbitrary $G$-invariant morphism $f : X \to Y$ factors through a canonical morphism $\pi : X \to X\sslash G$. Such a categorical quotient exists for instance when $G$ is reductive and $X$ is affine. But even in that case, the variety $X \sslash G$ is in general distinct from the orbit space $X/G$: to get a point in $X\sslash G$, one needs to further identify two orbits if their closures intersect, and this turns out to be equivalent to only keeping the closed orbits. More generally, to construct quotients of quasi-projective variety, one needs to impose a \textit{semistability condition} on the orbits, which selects an open subset $X^{\sst} \subseteq X$, of which one can then construct a categorical quotient $X^{\sst} \sslash G$ \cite{seshadri_quotientspacesmoduloreductivealgebraicgroups, seshadri_geometricreductivityoverarbitrarybase}. Moreover, there is a dense open subset of $X^{\sst} \sslash G$ which is an orbit space: its points are in bijection with the so-called \textit{stable} $G$-orbits in $X^{\sst}$.

\smallskip

As an application, one can construct moduli spaces of (semi)-stable vector bundles using GIT. Indeed, Mumford showed in \cite{mumf_projectiveinvar}, that for vector bundles over a smooth projective curve, the GIT-(semi)-stability condition translates to a much more concrete slope-(semi)-stability condition (Definition \ref{def_slopestability}). This was generalized by A. Ramanathan in \cite{ramanathan_stableprincipalbundles} to principal $G$-bundles for a reductive structure group $G$, where subbundles are replaced by parabolic reductions. When $G = \GL(r,\C)$, Ramanathan recovers Mumford's slope-(semi)-stability condition. Similarly, for vector bundles that are \textit{decorated} with additional structure, such as a symplectic form or a non-degenerate quadratic form, it is possible to formulate Ramanathan's semistability in terms of a slope condition, which is only tested on certain special subbundles (namely isotropic subbundles, in the case of symplectic and orthogonal vector bundles). For convenience, we provide an explicit proof of these facts in Theorems \ref{thm_ramversusslopesymp} and \ref{thm_ramversusslopeorth}, building on the approach of Hyeon and Murphy in  \cite{hyeon-murphy_noteonthestabilityofprincbund}. The point is that, for decorated vector bundles, stability and semistability can also be formulated using slope conditions that are tested only on the isotropic subbundles of these decorated vector bundles.

\smallskip

When studying vector bundles or principal bundles over an algebraic curve, a natural question is: How far is such a bundle from being semistable? For a vector bundle, the question is answered by looking at the Harder-Narasimhan filtration, first constructed in \cite{harder-narasimhan_onthecohomologygroupsofmodulispaces}. This filtration is uniquely determined by the condition that the successive quotients are semistable and have their slopes arranged in strictly decreasing order (see Theorem \ref{thm_hardernarasimhan} for a proof of the Harder-Narasimhan theorem). For a principal bundle, the Harder-Narasimhan filtrations must be replaced by a reduction of structure group to a so-called \textit{canonical reduction} to a uniquely determined parabolic subgroup. For complex reductive groups, canonical reductions were introduced by Atiyah and Bott in \cite[$\S$10]{atiyah-bott_yangmills}. Over an arbitrary base field, the construction is due to K.~Behrend in \cite{behrend_semi-stabilityofreductivegroupschemesovercurves} using complementary polyhedra (see also \cite{biswas-holla_harder-narasimhanreductionofprincbund}). Since semistability of symplectic and orthogonal vector bundles can be characterized using slope-semistability, it is natural to try and characterize canonical reductions of such bundles in terms of certain special filtrations, and we write this out explicitly in Theorems \ref{thm_ramversusslopesymp} and \ref{thm_ramversusslopeorth}. As we shall see, the parabolic subgroup associated to the canonical reduction of a principal $G$-bundle $\xi$ determines a filtration by isotropic sub-bundles of the decorated vector bundle associated to $\xi$ but, in order to relate this to the usual Harder-Narasimhan filtration of the underlying vector bundle, one needs to add the co-isotropic complements of these subbundles to the filtration. In our exposition, we follow the Atiyah-Bott approach and make systematic use of the special-orthogonal Harder-Narasimhan filtration of the adjoint bundle of a principal $G$-bundle (see Subsection \ref{subsec_ab}). Then in Subsection \ref{subsec_bh}, we compare this to the Biswas-Holla approach, in which the existence of a canonical reduction is proved using dominant characters. We unify the two approaches in Theorem \ref{thm_canonreductunique}. Important references for our exposition are the works of Friedman-Morgan \cite{friedman-morgan_ontheconversetoatheoremofatiyahbott} and Ho-Liu \cite{ho-liu_yangmillsconnectionsorientable}.

\smallskip

Since moduli problems can also be formulated in terms of moduli stacks, we can study the stratifications of the stack of principal bundles $Bun(G)$ induced by canonical reductions. This is related to the notion of $\Theta$-stratifications of moduli stacks introduced by D. Halpern-Leistner in \cite{danielhalpernleistner_onthestructureofinstabilityinmodulitheory} and is currently a topic of active investigation in moduli theory (see for instance \cite{alperDHLheinloth_existenceofmodulispacesforalgebraicstacks}). In Section \ref{sec_hntypes}, we show how canonical reductions induce strata of $Bun(G)$ labeled by Harder-Narasimhan types, and we make these types explicit for decorated vector bundles. When $G = \GL(r,\C)$, the Harder-Narasimhan type of a vector bundle $E$ consists essentially of the topological invariants (rank and degree) of the successive quotients of the Harder-Narasimhan filtration of $E$.

\section{Semistable bundles}\label{sec_semistab}

We compare the slope-semistability of vector bundles from Mumford in \cite{mumf_projectiveinvar} with the semistability of principal bundles from Ramanathan in \cite{ramanathan_stableprincipalbundles}. Following Hyeon-Murphy in \cite{hyeon-murphy_noteonthestabilityofprincbund}, we recall that they coincide when the structure group is $\GL(r,\C)$.

\subsection{Slope-semistability}\label{subsec_slopesemistab}

In \cite{mumf_projectiveinvar}, Mumford introduced slope-(semi)stability to study moduli problems of vector bundles, using geometric invariant theory.

\begin{definition}\label{def_slopestability}
    Let $E$ be a non-zero holomorphic vector bundle of rank $r$ on a compact Riemann surface $X$, then:
    \begin{enumerate}[label=(\alph*)]
        \item The \textit{slope} of $E$ is $\mu(E):=\deg(E)/\rk(E)$, where $\deg(E)$ is the degree of $E$ and $\rk(E)$ is the rank of $E$.

        \item $E$ is \textit{slope-semistable} if all non-zero subbundles $F\subsetneq E$ fulfill $\mu(F)\leq\mu(E)$.

        \item $E$ is \textit{slope-stable} if all non-zero proper subbundles $F\subsetneq E$ fulfill $\mu(F)<\mu(E)$.
    \end{enumerate}
\end{definition}

To generalize stability to principal $G$-bundles, where $G$ is a complex reductive group, we wish to rephrase it in terms of principal $\GL(r,\C)$-bundles. As $\GL(r,\C)$ is a reductive group, we consider the \textit{Cartan subgroup} $T$ of $\GL(r,\C)$ of diagonal matrices, and the \textit{Borel subgroup} $B$ of upper-right triangular matrices, such that $T\subseteq B\subseteq\GL(r,\C)$. We call those the \textit{standard} Cartan and Borel subgroups of $\GL(r, \C)$. A \textit{standard parabolic} subgroup $P\subseteq\GL(r,\C)$ is a subgroup that contains $B$. Such subgroups consist of upper-right block matrices and correspond uniquely to subsets $I$ of the set of \textit{simple roots} $\triangle$ induced by $B$:
\begin{equation}\label{eq_simpleroots}
    \triangle=\left\{\alpha_{l,l+1}\in\Lie(T)^*\ \middle\vert\ \alpha_{l,l+1}(A):=A_{l,l}-A_{l+1,l+1}\right\}
\end{equation}
where $A = (A_{ij})_{1 \leqslant i, j \leqslant r}$ for every $r \times r$ matrix with coefficients in $\C$. We henceforth index standard parabolics as $P_I$, and refer to \cite{borel_linearalgebraicgroups} for a review of Cartan, Borel, and parabolic subgroups of reductive groups.

\begin{remark}\label{rem_subbundcocy}
    Let $F\subsetneq E$ be a non-zero subbundle of rank $l$.
    There exist cocycles $(\sigma_{i,j})_{i,j\in J}$ of $E$, subordinate to a covering $(U_i)_{i\in J}$ of $X$, of the form:
    \begin{equation*}
        \sigma_{i,j}:U_i\cap U_j\rightarrow \GL(r,\C),\quad x\mapsto\begin{pmatrix}
            \alpha_{i,j}(x) & \beta_{i,j}(x)  \\
            0               & \delta_{i,j}(x) \\
        \end{pmatrix},\quad i,j\in J,
    \end{equation*}
    where $(\alpha_{i,j})_{i,j\in J}$ are cocycles of $F$ and $(\delta_{i,j})_{i,j\in J}$ are cocycles of $E/F$. Thus, $(\sigma_{i,j})_{i,j\in J}$ maps into a standard parabolic $P_I$, where $I=\{\alpha_{l,l+1}\}\subseteq\triangle$.
\end{remark}

Using the same cocycles as in Remark \ref{rem_subbundcocy}, we can identify subbundles $F\subsetneq E$ with \textit{reductions} of the \textit{frame bundle} $\Fr(E)$ of $E$ to a structure group $P_I$, which is a \textit{standard maximal parabolic} subgroup of $\GL(r,\C)$. Reductions are sections $s:X\rightarrow\Fr(E)/P_I$ of the bundle $\Fr(E)/P_I$.

\begin{lemma}\label{lem_reductsubbundcorr}
    The following are in correspondence:
    \begin{enumerate}[label=(\roman*)]
        \item Subbundles $F\subsetneq E$ of rank $l$.

        \item Reductions $s:X\rightarrow \Fr(E)/P_I$ of $\Fr(E)$ to $P_I$, where $I=\{\alpha_{l,l+1}\}\subseteq\triangle$.
    \end{enumerate}
\end{lemma}

\begin{proof}
    For \textit{(i)} to \textit{(ii)}, we induce a section $s:X\rightarrow\Fr(E)/P_I$, $x\mapsto s(x)$, where $s(x)$ represents the isomorphisms $\C^r\rightarrow E_x$ that restrict to an isomorphism $\C^l\times\{0\}^{r-l}\rightarrow F_x$.

    For \textit{(ii)} to \textit{(i)}, for all $x\in X$, $F_x=s(x)(\C^l\times\{0\}^{r-l})$ is a well-defined subspace of $E_x$. Since $s$ is holomorphic, the subspaces $(F_x)_{x\in X}$ glue to a subbundle $F\subsetneq E$.
\end{proof}

For each subbundle $F$, the corresponding section $s:X\rightarrow \Fr(E)/P_I$ gives frames of $E$ that extend frames of $F$. Given local trivializations of $E$, we can induce from $s$ a family of cocycles $(\sigma_{i,j})_{i,j\in J}$ of $E$, compatible with $F$ as in Remark \ref{rem_subbundcocy}. Through Lemma \ref{lem_reductsubbundcorr}, the natural way to generalize slope-stability to principal $G$-bundles is by imposing slope conditions on reductions to standard maximal parabolics $P$ of $G$. For a \textit{decorated vector bundle} $(E,s)$, i.e., a vector bundle equipped with a reduction $s:X\rightarrow\Fr(E)/G$ of structure group $G\subseteq \GL(r,\C)$, it means that the slope conditions from Definition \ref{def_slopestability} are tested only on certain subbundles $F\subsetneq E$ that respect the decoration.
We first work out the $G=\SL(r,\C)$ case, corresponding to vector bundles with trivial determinants, i.e., special vector bundles.

\begin{definition}\label{def_specialvecbun}
    A \textit{special vector bundle} $(E,\tau)$ consists of a vector bundle $E$, and a non-vanishing section $\tau:X\rightarrow\det(E)$ of the determinant bundle.
\end{definition}

For a rank $r$ special vector bundle $(E,\tau)$, giving a section $\tau$ is equivalent to giving a reduction $s:X\rightarrow\Fr(E)/\SL(r,\C)$ of the frame bundle. Hence, a special vector bundle is a decorated vector bundle. Just as vector bundles $E$, of rank $r$, correspond to principal $\GL(r,\C)$-bundles through their frame bundles $\Fr(E)$, special vector bundles $(E,\tau)$, of rank $r$, correspond to principal $\SL(r,\C)$-bundles through their \textit{special frame bundles} $\Fr_{\SL}(E,\tau)$.

\begin{definition}\label{def_specialframebun}
    For a special vector bundle $(E,\tau)$ of rank $r$, the \textit{special frame bundle} $\Fr_{\SL}(E,\tau)$ is a principal $\SL(r,\C)$-bundle, whose fiber at $x\in X$ is given by:
    \begin{equation*}
        \SIso_\C^\tau(\C^{r},E_x):=\left\{f\in\Iso_\C(\C^{r},E_x)\ \middle\vert\ \det(f)=\tau(x)\right\}.
    \end{equation*}
    The right $\SL(r,\C)$-action on the fiber of $x\in X$ is given by:
    \begin{equation*}
        \SIso_\C^\tau(\C^{r},E_x)\times\SL(r,\C)\rightarrow\SIso_\C^\tau(\C^{r},E_x),\quad(f,A)\mapsto [v\mapsto f(Av)].
    \end{equation*}
\end{definition}

Just as in the $\GL(r,\C)$ case, $\SL(r,\C)$ is a reductive group (in fact it is semisimple). We intersect the standard Cartan and Borel subgroups of $\GL(r,\C)$ with $\SL(r,\C)$ to obtain Cartan and Borel subgroups of $\SL(r,\C)$.
The root system of $\SL(r,\C)$ is the same as that of $\GL(r,\C)$, so standard maximal parabolic subgroups $P_I$ of $\SL(r,\C)$ correspond to subsets $I=\{\alpha_{l,l+1}\}\subseteq\triangle$.
This leads to a variation of Lemma \ref{lem_reductsubbundcorr} for special vector bundles.

\begin{lemma}\label{lem_reductsubbundcorrsl}
    The following are in correspondence:
    \begin{enumerate}[label=(\roman*)]
        \item Subbundles $F\subsetneq E$ of $(E,\tau)$ of rank $l$.

        \item Reductions $s:X\rightarrow \Fr_{\SL}(E,\tau)/P_I$ of $\Fr_{\SL}(E,\tau)$ to $P_I$, where $I=\{\alpha_{l,l+1}\}\subseteq\triangle$.
    \end{enumerate}
\end{lemma}

Following the discussion after Lemma \ref{lem_reductsubbundcorr}, stability for vector bundles $E$ and special vector bundles $(E,\tau)$ should be equivalent since Lemmas \ref{lem_reductsubbundcorr} and \ref{lem_reductsubbundcorrsl} give the analogous correspondences.
The decoration of a reduction of $E$ to $\SL(r,\C)$ does not restrict the set of subbundles $F\subsetneq E$ with respect to which one tests the slope conditions. This means that the stability and semistability conditions for decorated vector bundles $(E, \tau)$, where $\tau$ is a non-vanishing section of $\det(E)$, are exactly that of stability and semistability of the underlying vector bundle $E$. Observe that $\SL(r,\C)$ is the \textit{commutator subgroup} $[\GL(r,\C),\GL(r,\C)]$ of $\GL(r,\C)$. In Lemma \ref{lem_derived}, we will see the same phenomenon as Lemma \ref{lem_reductsubbundcorrsl}, that the stability of principal $[G,G]$-bundles and the stability of their extensions to principal $G$-bundles are equivalent. Note that the commutator subgroup $[G,G]$ is always semisimple and has the same root system as $G$.

\subsection{Ramanathan-semistability}

We review Ramanathan's work in \cite{ramanathan_stableprincipalbundles}, including his definition of semistability and stability for principal $G$-bundles, where $G$ is complex and reductive. Let $\xi$ be a holomorphic principal $G$-bundle on a compact Riemann surface $X$.
For a reduction $s:X\rightarrow\xi/H$ of $\xi$ to a subgroup $H\subseteq G$, we can pull the bundle $\xi\rightarrow\xi/H$ back by $s$ to obtain a principal $H$-bundle $s^*\xi$ on $X$. We also refer to $s^*\xi$ as the \textit{reduction} of $\xi$ to $H$. Following \cite{ramanathan_stableprincipalbundles}, we define the \textit{vertical tangent bundle} $V_{\xi/H}$, which is the vector bundle on $\xi/H$ consisting of tangent vectors along the fibers of $\xi/H\rightarrow X$.

\begin{definition}\label{def_stabsstprincbund}
    Let $\xi$ be a holomorphic principal $G$-bundle on a compact Riemann surface $X$.
    \begin{enumerate}[label=(\alph*)]
        \item $\xi$ is \textit{Ramanathan-semistable} if for all maximal parabolic subgroups $P$ of $G$, and for all reductions $s : X \to \xi/P$, we have $\deg(s^*V_{\xi/P})\geq0$.

        \item $\xi$ is \textit{Ramanathan-stable} if for all maximal parabolic subgroups $P$ of $G$, and for all reductions $s^*\xi$ of $\xi$ to $P$, we have $\deg(s^*V_{\xi/P})>0$.
    \end{enumerate}
\end{definition}

The following two remarks give us concrete ways to verify Ramanathan-semistability and Ramanathan-stability. When statements for semistability and stability are analogous, we adopt the convention of writing (semi)-stability for brevity, along with inequalities $(\leq)$ and $(\geq)$.

\begin{remark}\label{rem_adbundleexactseq} Let $s^*\xi$ be a reduction of $\xi$ to a parabolic subgroup $P$ of $G$. The adjoint representation on $P$ induces a short exact sequence of vector bundles:
    \begin{equation*}
        0\rightarrow\ad(s^*\xi)\rightarrow\ad(\xi)\rightarrow s^*V_{\xi/P}\rightarrow0,
    \end{equation*}
    as seen in \cite[Remark 2.2]{ramanathan_stableprincipalbundles}. Since $G$ is reductive, there exists a nondegenerate adjoint-invariant form on $\Lie(G)$ such that $\ad(\xi)$ is self-dual, so $\ad(\xi)$ has degree $0$. From the short exact sequence, we have $\deg(s^*V_{\xi/P})=-\deg(\ad(s^*\xi))$. This can be used as an alternate characterization of Ramanathan-(semi)-stability. Namely, $\xi$ is Ramanathan-(semi)-stable when for all reductions $s:X\rightarrow\xi/P$ to maximal parabolic subgroups $P$ of $G$, we have $\deg(\ad(s^*\xi)) (\leq) 0$.
\end{remark}

\begin{remark}\label{rem_stabsstprincbund}
    Once we fix a Borel subgroup $B$, to verify Ramanathan-(semi)-stability, it suffices to verify the inequality for reductions to maximal \textit{standard} parabolic subgroups of $G$. For a principal $G$-bundle $\xi$ with this property and a reduction $s:X\rightarrow\xi/P$ to any maximal parabolic subgroup $P\subsetneq G$, we have that:
    \begin{enumerate}[label=(\roman*)]
        \item There exists $g\in G$, such that $gPg^{-1}=P_I$ for some standard maximal parabolic subgroup $P_I\subsetneq G$. See \cite[IV. 11.1 Theorem, IV. 11.3 Corollary]{borel_linearalgebraicgroups}.

        \item We can conjugate a reduction $s:X\rightarrow\xi/P$ by $g$ to obtain a reduction $s':X\rightarrow\xi/P_I$.

        \item $s^*\xi$ and $s'^*\xi$ share cocycles up to conjugation, so the same is true for $\ad(s^*\xi)$ and $\ad(s'^*\xi)$. Thus, the determinant bundles $\det(\ad(s^*\xi))$ and $\det(\ad(s'^*\xi))$ are isomorphic, and the adjoint bundles have the same degree, i.e., $\deg(\ad(s^*\xi))=\deg(\ad(s'^*\xi))(\leq) 0$.
    \end{enumerate}
    Through these points, the claim of this remark follows from Remark \ref{rem_adbundleexactseq}.
\end{remark}

The next lemma is from \cite[Proposition 1]{hyeon-murphy_noteonthestabilityofprincbund}, which we use to recall the equivalence between slope-(semi)-stability and Ramanathan-(semi)-stability. Let $E$ be a vector bundle of rank $r$.

\begin{lemma}\label{lem_tgpisom}
    For a subbundle $F\subsetneq E$ of rank $l$, and the corresponding reduction $s:X\rightarrow\Fr(E)/P_I$ from Lemma \ref{lem_reductsubbundcorr}, we have:
    \begin{equation*}
        s^*V_{\Fr(E)/P_I}\cong F^*\otimes(E/F).
    \end{equation*}
\end{lemma}

\begin{proof}
    Let $(\sigma_{i,j})_{i,j\in J}$ be cocycles of $E$ compatible with $F$, as in Remark \ref{rem_subbundcocy}, then the maps $x\mapsto (\alpha_{i,j}(x)^{-1})^T\otimes\delta_{i,j}(x)$ define cocycles of $F^*\otimes(E/F)$.
    Thus, to show that $s^*V_{\Fr(E)/P_I}\cong F^*\otimes(E/F)$, it suffices to prove that $((\alpha_{i,j}^{-1})^T \otimes \delta_{i,j})_{i,j\in J}$ are cocycles of $s^*V_{\Fr(E)/P_I}$.
    We can construct cocycles $(\tau_{i,j})_{i,j\in J}$ of $s^*V_{\Fr(E)/P_I}$ by composing $(\sigma_{i,j})_{i,j\in J}$ with $\Ad:P_I\rightarrow\GL(\Lie(\GL(r,\C))/\Lie(P_I))$, and calculate:
    \begin{align*}
        \Ad\begin{pmatrix}
               a & b \\
               0 & d \\
           \end{pmatrix}\begin{pmatrix}
                            * & * \\
                            C & * \\
                        \end{pmatrix} & =\begin{pmatrix}
                                             a & b \\
                                             0 & d \\
                                         \end{pmatrix}\begin{pmatrix}
                                                          * & * \\
                                                          C & * \\
                                                      \end{pmatrix}\begin{pmatrix}
                                                                       a & b \\
                                                                       0 & d \\
                                                                   \end{pmatrix}^{-1} \\
                                     & =\begin{pmatrix}
                                            a & b \\
                                            0 & d \\
                                        \end{pmatrix}\begin{pmatrix}
                                                         * & * \\
                                                         C & * \\
                                                     \end{pmatrix}\begin{pmatrix}
                                                                      a^{-1} & *      \\
                                                                      0      & d^{-1} \\
                                                                  \end{pmatrix}     \\
                                     & =\begin{pmatrix}
                                            *        & * \\
                                            dCa^{-1} & * \\
                                        \end{pmatrix}
    \end{align*}
    Hence, the cocycles $(\tau_{i,j})_{i,j\in J}$ are of the form $x\mapsto[C\mapsto\delta_{i,j}(x)C\alpha_{i,j}(x)^{-1}]$, alternatively $x\mapsto (\alpha_{i,j}(x)^{-1})^T\otimes\delta_{i,j}(x)$, as we wanted to show.
\end{proof}

We can now prove the following theorem, from \cite[Corollary 1]{hyeon-murphy_noteonthestabilityofprincbund}.

\begin{theorem}\label{thm_ramversusslope}
    For a vector bundle $E$ of rank $r$, the following are equivalent:
    \begin{enumerate}[label=(\roman*)]
        \item $\Fr(E)$ is Ramanathan-(semi)-stable.

        \item $E$ is slope-(semi)-stable.
    \end{enumerate}
\end{theorem}

\begin{proof}
    Given a subbundle $F\subsetneq E$ of rank $l$, we have through Lemma \ref{lem_tgpisom}:
    \begin{align*}
        \deg(s^*V_{Fr(E)/P_I}) & = \deg(F^*\otimes(E/F))           \\
                               & =  -\deg(F)(r-l)+\deg(E/F)l       \\
                               & = -\deg(F)(r-l)+\deg(E)l-\deg(F)l \\
                               & = -\deg(F)r+\deg(E)l
    \end{align*}
    Thus, $\deg(s^*V_{\Fr(E)/P_I})(\geq)0$ is equivalent to $\mu(F)(\leq)\mu(E)$.
\end{proof}

Following the discussion after Lemma \ref{lem_reductsubbundcorrsl}, the slope-(semi)-stability of a special vector bundle should be equivalent to the Ramanathan-(semi)-stability its special frame bundle. Indeed, we get the following statement.

\begin{theorem}\label{thm_ramversusslopespecial}
    For a special vector bundle $(E,\tau)$ of rank $r$, the following are equivalent:
    \begin{enumerate}[label=(\roman*)]
        \item $\Fr_\SL(E,\tau)$ is Ramanathan-(semi)-stable.

        \item $E$ is slope-(semi)-stable.
    \end{enumerate}
\end{theorem}

We can prove Theorem \ref{thm_ramversusslopespecial} in the same way as Theorem \ref{thm_ramversusslope}. Alternatively, we can invoke the following lemma, which generalizes the connection between Theorems \ref{thm_ramversusslope} and \ref{thm_ramversusslopespecial} and shows why the notion of stability and semistability of a principal $[G,G]$-bundle coincides with that of the associated $G$-bundle.

\begin{lemma}\label{lem_derived}
    For a principal $[G,G]$-bundle $\xi_1$, the following are equivalent:
    \begin{enumerate}[label=(\roman*)]
        \item $\xi_1$ is Ramanathan-(semi)-stable.

        \item The extension $\xi=\xi_1(G)$ through the inclusion $[G,G]\hookrightarrow G$ is Ramanathan-(semi)-stable.
    \end{enumerate}
\end{lemma}

\begin{proof}
    Let $P_1$ be a maximal parabolic subgroup of $[G,G]$. Let $P=P_1\Rad(G)$ be the subgroup of $G$ generated by $P_1$ and the \textit{radical} $\Rad(G)$ of $G$, as defined in \cite[IV. 11.21]{borel_linearalgebraicgroups}, then $P$ is a maximal parabolic subgroup of $G$, since $G=[G,G]\Rad(G)$ following \cite[IV. 14.2 Proposition]{borel_linearalgebraicgroups}, and the fact that $\Rad(G)$ is the connected component of the center $\Zbf(G)$ from \cite[IV. 11.21 Proposition]{borel_linearalgebraicgroups}.
    Furthermore, we have a $[G,G]$-equivariant isomorphism of complete varieties:
    \begin{equation}\label{eq_derived}
        G/P=([G,G]P)/P\cong [G,G]/([G,G]\cap P)=[G,G]/P_1.
    \end{equation}

    Let $s_1:X\rightarrow\xi_1/P_1$ be a reduction of $\xi_1$ to $P_1$.
    Using the isomorphism in (\ref{eq_derived}), we induce a $[G,G]$-equivariant isomorphism $\xi_1/P_1\cong\xi/P$, so the section $s_1:X\rightarrow\xi_1/P_1$ corresponds to a section $s:X\rightarrow\xi/P$. Thus,
    $s^*\xi$ is isomorphic to an extension of $s_1^*\xi_1$ through $P_1\hookrightarrow P$, and both share the same cocycles $(\sigma_{i,j})_{i,j\in J}$ that map into $P_1$.

    Using that $\Lie(P_1)=\Lie([G,G])\cap\Lie(P)$, and that $\Lie(P)=\Lie(P_1)\oplus\Lie(\Rad(G))$ as Lie algebras, we claim that for all $i,j\in J$, the following diagram commutes:
    \begin{equation}
        \begin{aligned}\label{eq_deriveddiagram}
            \begin{tikzpicture}[scale=1.3]
                \node (A) at (-1.5,1) {$U_i\cap U_j$};
                \node (B) at (0,1) {$P_1$};
                \node (C) at (1.5,1) {$\GL(\Lie(P_1))$};
                \node (D) at (3,1) {$\C^\times$};
                \node(E) at (0,0) {$P$};
                \node(F) at (1.5,0) {$\GL(\Lie(P))$};
                \node(G) at (3,0) {$\C^\times$};
                \path[->,font=\scriptsize,>=angle 90]
                (A) edge node[above]{$\sigma_{i,j}$} (B)
                (B) edge node[above]{$\Ad$} (C)
                (C) edge node[above]{$\det$} (D)
                (A) edge node[above]{$\sigma_{i,j}$} (E)
                (E) edge node[above]{$\Ad$} (F)
                (F) edge node[above]{$\det$} (G);
                \draw[right hook->] (B) -- (E);
                \draw[double,double distance between line centers=0.2em] (D) -- (G);
            \end{tikzpicture}
        \end{aligned}
    \end{equation}

    Since the adjoint representation of $P_1$ acts on $\Lie(\Rad(G))$ through the identity, the morphisms on the bottom row of (\ref{eq_deriveddiagram}) have the same determinants as the top row of (\ref{eq_deriveddiagram}). Hence, the diagram commutes.

    Since the top row of (\ref{eq_deriveddiagram}) gives us cocycles of $\det(\ad(s_1^*\xi_1))$, and since the bottom row of (\ref{eq_deriveddiagram}) gives us cocycles of $\det(\ad(s^*\xi))$, these two line bundles are isomorphic, so $\deg(\ad(s_1^*\xi_1))=\deg(\ad(s^*\xi))$. Using Remark \ref{rem_adbundleexactseq}, the claim follows.
\end{proof}

Lastly, we recall an important use of Remark \ref{rem_adbundleexactseq}. The adjoint bundle of a principal $G$-bundle determines its Ramanathan-semistability. To prove this, we use \cite[Theorem 3.18]{ramanan-ramanathan_someremarksoninstabilityflag}.

\begin{theorem}\label{thm_instflag}
    Let $\varphi:G\rightarrow H$ be a homomorphism of complex reductive groups that maps the radical $\Rad(G)$ into $\Rad(H)$. For a Ramanathan-semistable principal $G$-bundle $\xi$, the extension $\xi(H)$ of $\xi$ to $H$ is Ramanathan-semistable.
\end{theorem}

This theorem, proven in \cite{ramanan-ramanathan_someremarksoninstabilityflag} by defining admissible reductions, will be helpful to us in later sections as well.
We now present the result of
\cite[3.18 Corollary]{ramanathan_moduliforprincipalbundlesI}.

\begin{lemma}\label{lem_adbundle}
    A principal $G$-bundle $\xi$ is Ramanathan-semistable if and only if its adjoint bundle $\ad(\xi)$ is slope-semistable.
\end{lemma}

\begin{proof}
    The forward implication is a consequence of Theorem \ref{thm_instflag}: the radical of a reductive group is a subgroup of the center, from \cite[IV. 11.21 Proposition]{borel_linearalgebraicgroups}, and therefore $\Rad(G)$ gets mapped to $0$ under the adjoint representation, whose kernel is $\Zbf(G)$. Ramanathan-semistability is preserved in the frame bundle of the adjoint bundle $\ad(\xi)$, which is equivalent to slope-semistability due to Theorem \ref{thm_ramversusslope}.

    For the reverse implication, let $s^*\xi$ be a reduction of $\xi$ to a maximal parabolic subgroup $P$ of $G$. Due to Remark \ref{rem_adbundleexactseq}, it suffices to show $\deg(\ad(s^*\xi))\leq 0$. Since $\ad(\xi)$ has degree $0$, we have $\mu(\ad(s^*\xi))\leq \mu(\ad(\xi))=0$, as $\ad(s^*\xi)$ can be viewed as a subbundle of $\ad(\xi)$. From this, $\deg(\ad(s^*\xi))\leq 0$ follows.
\end{proof}

The reverse implication of Lemma \ref{lem_adbundle} is also true for stability, as we can replace the inequalities with strict inequalities in the proof above. However, the same is not true for the forward implication. In \cite[Lemma 2]{hyeon-murphy_noteonthestabilityofprincbund}, it is explained that for stability, the forward implication is already false if $G$ is not semisimple, as in the case of $G=\GL(r,\C)$.

\section{Decorated vector bundles}\label{sec_decorated}

Recall that a \textit{decorated vector bundle} $(E,s)$ is a vector bundle $E$ equipped with a reduction $s$ to a complex reductive subgroup $G \subseteq \GL(r;\C)$. Examples include symplectic and special-orthogonal vector bundles, as well as vector bundles with trivial determinants, i.e., special vector bundles. As we did for special vector bundles in Theorem \ref{thm_ramversusslopespecial}, we recall conditions for symplectic and special-orthogonal vector bundles to be Ramanathan-(semi)-stable in Theorems \ref{thm_ramversusslopesymp} and \ref{thm_ramversusslopeorth}.

\subsection{Symplectic vector bundles}\label{subsec_sympvecbun}
\begin{definition}\label{def_sympbund}
    Let $E$ be a holomorphic vector bundle of rank $r$ on a compact Riemann surface $X$, then:
    \begin{enumerate}[label=(\alph*)]
        \item A \textit{symplectic vector bundle} $(E,\beta)$ is a vector bundle $E$ admitting a symplectic form $\beta: E \oplus  E \to X\times\C$, i.e., a nondegenerate antisymmetric form.

        \item For a symplectic vector bundle $(E,\beta)$, a subbundle $F\subseteq E$ is:
              \begin{enumerate}[label=(b.\alph*)]
                  \item \textit{isotropic} if for all $x\in X$, $F_x\subseteq F_x^\perp:=\{v\in E_x\ \vert\ \forall w\in F_x:\beta_x(v,w)=0\}$.

                  \item \textit{co-isotropic} if for all $x\in X$, $F_x^\perp\subseteq F_x$.

                  \item \textit{Lagrangian} if for all $x\in X$, $F_x =F_x^\perp$.
              \end{enumerate}
    \end{enumerate}
\end{definition}

We claim that the Ramanathan-(semi)-stability of such bundles amounts to checking slope conditions on isotropic subbundles. To see this, we need to understand the root system of the \textit{symplectic group} $\Sp(2n,\C)$, where $r=2n$. In most literature, $\Sp(2n,\C)$ is defined as:
\begin{equation*}
    \left\{A\in\GL(2n,\C)\ \middle\vert\ A^TW_{2n}A=W_{2n}\right\},
\end{equation*}
where $W_n$ represents the standard symplectic form $\omega(x,y):=\sum_{i=1}^n x_{n+i}y_i-x_iy_{n+i}$, i.e.:
\begin{align}
    \nonumber W_{2n} & :=\begin{pmatrix}
                             0    & I_n \\
                             -I_n & 0
                         \end{pmatrix}\in\GL(2n,\C),\hspace{3em} I_n:=\begin{pmatrix}
                                                                          1 & 0                                   & 0 \\
                                                                          0 & \rotatebox[origin=c]{-35}{$\cdots$} & 0 \\
                                                                          0 & 0                                   & 1
                                                                      \end{pmatrix}\in\GL(n,\C).         \\
    \intertext{However, for our purposes, we define $\Sp(2n,\C)$ as:}
    \nonumber        & \rlap{\hspace{2.8em}$\Sp(2n,\C):=\left\{A\in\GL(2n,\C)\ \middle\vert\ A^TM_{2n}A=M_{2n}\right\},$} \\
    \intertext{where $M_{2n}$ represents the symplectic form $\omega(x,y):=\sum_{i=1}^n x_{2n+1-i}y_i-x_iy_{2n+1-i}$, i.e.:}
    M_{2n}           & :=\begin{pmatrix}
                             0      & K_{n} \\
                             -K_{n} & 0
                         \end{pmatrix}\in\GL(2n,\C),\hspace{1.93em} K_n:=\begin{pmatrix}
                                                                             0 & 0                                  & 1 \\
                                                                             0 & \rotatebox[origin=c]{35}{$\cdots$} & 0 \\
                                                                             1 & 0                                  & 0
                                                                         \end{pmatrix}\in\GL(n,\C). \label{eq_m2n}
\end{align}

Both variants are isomorphic by conjugation in $\GL(2n,\C)$. Our variant allows us to intersect our earlier choices of Cartan and Borel subgroups of $\GL(2n,\C)$ to obtain Cartan and Borel subgroups of $\Sp(2n,\C)$. This has the advantage that all standard parabolic subgroups of $\Sp(2n,\C)$ appear as upper-right block matrices, allowing us to more easily give symplectic versions of
Lemmas \ref{lem_reductsubbundcorr} and \ref{lem_tgpisom}. However, the root system of $\Sp(2n,\C)$ is quite different to that of $\GL(2n,\C)$ and $\SL(2n,\C)$. Standard parabolic subgroups of $\Sp(2n,\C)$ are in correspondence with subsets $I$ of the set of simple roots $\triangle$:
\begin{equation}\label{eq_simplerootssp}
    \triangle=\left\{\alpha_{l,l+1}\in\Lie(T)^*\ \middle\vert\ 1\leq l \leq n-1\right\}\sqcup\left\{2\alpha_n\in\Lie(T)^*\ \middle\vert\ \alpha_n(A):=A_{n,n}\right\},
\end{equation}
using notation from (\ref{eq_simpleroots}).

\begin{example}\label{ex_sp4}
    The symplectic group $\Sp(4,\C)$ has the simple roots $\triangle=\{\alpha_{1,2},2\alpha_2\}$. The nontrivial standard parabolic subgroups are:
    \begin{align*}
        P_I & =\left\{\begin{pmatrix}
                          * & * & * & * \\
                          0 & * & * & * \\
                          0 & * & * & * \\
                          0 & 0 & 0 & * \\
                      \end{pmatrix}\in\Sp(4,\C)\right\}, & I & =  \{\alpha_{1,2}\}, \\
        P_I & =\left\{\begin{pmatrix}
                          * & * & * & * \\
                          * & * & * & * \\
                          0 & 0 & * & * \\
                          0 & 0 & * & * \\
                      \end{pmatrix}\in\Sp(4,\C)\right\}, & I & =  \{2\alpha_2\}.
    \end{align*}
\end{example}
Note the shape of these matrices: The sizes of the diagonal blocks are mirrored from upper-left to lower-right. Generally, the shape of off-center blocks is determined by simple roots $\alpha_{l,l+1}$, whereas blocks split down the middle are determined by the simple root $2\alpha_n$. As we did for special vector bundles in Definition \ref{def_specialframebun}, we construct symplectic frame bundles.

\begin{definition}\label{def_specialsympframebun}
    For a symplectic vector bundle $(E,\beta)$, a \textit{symplectic frame bundle} $\Fr_\Sp(E,\beta)$ is a principal $\Sp(2n,\C)$-bundle, whose fiber at $x\in X$ is given by:
    \begin{equation*}
        \Symp_{\C}^\beta(\C^{2n},E_x),
    \end{equation*}
    consisting of isomorphisms $\C^{2n}\rightarrow E_x$ compatible with $\beta$ on $E_x$, and $\omega$ on $\C^{2n}$ induced by $M_{2n}$ from (\ref{eq_m2n}).
\end{definition}

We obtain a symplectic variant of Lemmas \ref{lem_reductsubbundcorr} and \ref{lem_reductsubbundcorrsl}.
\begin{lemma}\label{lem_reductsubbundcorrsymp}
    The following are in correspondence:
    \begin{enumerate}[label=(\roman*)]
        \item Isotropic subbundles $F\subsetneq E$ of $(E,\beta)$ of rank $l$.

        \item Reductions $s^*\Fr_\Sp(E,\beta)$ of $\Fr_\Sp(E,\beta)$ to $P_I$, where:
              \begin{align*}
                  I & =\{\alpha_{l,l+1}\}\subseteq\triangle, & 1 & \leq l \leq n-1 , \\
                  I & =\{2\alpha_n\}\subseteq\triangle,      & l & = n.
              \end{align*}
    \end{enumerate}
\end{lemma}

\begin{proof}
    \sloppy
    For \textit{(i)} to \textit{(ii)}, we induce a section $s:X\rightarrow\Fr_\Sp(E,\beta)/P_I$, $x\mapsto s(x)$, where $s(x)$ represents the symplectic isomorphisms $\C^{2n}\rightarrow E_x$ that restrict to symplectic isomorphisms $\C^{2n-l}\times\{0\}^l\rightarrow F_x^\perp$ and $\C^l\times\{0\}^{2n-l}\rightarrow F_x$.

    For \textit{(ii)} to \textit{(i)}, we argue analogously to Lemma \ref{lem_reductsubbundcorr} to construct an isotropic subbundle $F\subsetneq E$ with the isotropic fibers $F_x=s(x)(\C^l\times\{0\}^{2n-l})\subseteq s(x)(\C^{2n-l}\times\{0\}^l)=F^\perp_x$ for all $x\in X$.
\end{proof}

We then obtain a symplectic variant of Lemma \ref{lem_tgpisom}, allowing us to compare slope-(semi)-stability and Ramanathan-(semi)-stability later.

\begin{lemma}\label{lem_tgpisomsymp}
    For an isotropic subbundle $F\subsetneq E$ of $(E,\beta)$ of rank $l$, and the corresponding reduction $s:X\rightarrow\Fr_\Sp(E)/P_I$ from Lemma \ref{lem_reductsubbundcorrsymp}, we have:
    \begin{align}
        \det(s^*V_{\Fr_\Sp(E,\beta)/P_I}) & \cong\det(F^*\otimes (F^\perp/F))\otimes\left(\det(F^*)^{\otimes(l+1)}\right), & F & \neq F^\perp, \label{eq_tgpisomorth1} \\
        \det(s^*V_{\Fr_\Sp(E,\beta)/P_I}) & \cong\det(F^*)^{\otimes(l+1)},                                                 & F & =F^\perp. \label{eq_tgpisomorth2}
    \end{align}
\end{lemma}

\begin{proof}
    \sloppy
    Due to Remark \ref{rem_adbundleexactseq}, given cocycles $(\sigma_{i,j})_{i \in J}$ of $s^*\Fr_\Sp(E)/P_I$, we have cocycles $(\tau_{i,j})_{i \in J}=(\Ad\circ\sigma_{i,j})_{i \in J}$ of $s^*V_{\Fr_\Sp(E,\beta)/P_I}$, where $\Ad:P_I\rightarrow\GL(\Lie(\Sp(2n,\C))/\Lie(P_I))$ is the adjoint representation.

    We want to describe this representation. Assuming $F\neq F^\perp$, and using notation from (\ref{eq_m2n}), $\Lie(\Sp(2n,\C))/\Lie(P_I)$ consists of classes of matrices:
    \begin{align*}
        N_{D,G} & =\begin{pmatrix}
                       * & *               & * \\
                       D & *               & * \\
                       G & K_lD^TM_{2n-2l} & * \\
                   \end{pmatrix},\quad \begin{matrix}
                                           D\in\Mat((2n-2l)\times l,\C),                                    \\
                                           G\in \mathbf{W}:=\{B\in\Mat(l\times l,\C)\ \vert\ K_lB=B^TK_l\}, \\
                                       \end{matrix}
    \end{align*}
    where $D$ and $G$ have no relations. The group $P_I$ consists of matrices of the form:
    \begin{equation*}
        M_{a,b,c,d} =\begin{pmatrix}
            a & b & c                                      \\
            0 & d & -M_{2n-2l}(d^{-1})^T(b^T)(a^{-1})^TK_l \\
            0 & 0 & K_l(a^{-1})^TK_l                       \\
        \end{pmatrix},\hspace{0.1cm}  \begin{matrix}
            a\in\GL(l,\C),\hspace{0.08cm} b\in\Mat(s\times(2n-2l),\C), \\
            c\in\Mat(l\times l,\C),\hspace{0.15cm} d\in\Sp(2n-2l,\C).  \\
        \end{matrix}
    \end{equation*}
    with relations between $a$, $b$, $c$, and $d$ (which are not relevant to the proof). Using this, we can get an explicit description of $\Ad:P_I\rightarrow\GL(\Lie(\Sp(2n,\C))/\Lie(P_I))$:
    \begin{align}
        \nonumber\Ad                & (M_{a,b,c,d})(N_{D,G})                                              \\*
        \label{eq_tgpisomorthstart} & =\begin{pmatrix}
                                           \ldots & \ldots & \ldots                                 \\
                                           0      & d      & -M_{2n-2l}(d^{-1})^T(b^T)(a^{-1})^TK_l \\
                                           0      & 0      & K_l(a^{-1})^TK_l                       \\
                                       \end{pmatrix}\begin{pmatrix}
                                                        * & *      & * \\
                                                        D & *      & * \\
                                                        G & \ldots & * \\
                                                    \end{pmatrix}\begin{pmatrix}
                                                                     a^{-1} & \ldots & \ldots \\
                                                                     0      & \ldots & \ldots \\
                                                                     0      & 0      & \ldots \\
                                                                 \end{pmatrix}           \\
        \nonumber                   & =\begin{pmatrix}
                                           *                                         & *      & * \\
                                           dD-M_{2n-2l}(d^{-1})^T(b^T)(a^{-1})^TK_lG & *      & * \\
                                           K_l(a^{-1})^TK_lG                         & \ldots & * \\
                                       \end{pmatrix}\begin{pmatrix}
                                                        a^{-1} & \ldots & \ldots \\
                                                        0      & \ldots & \ldots \\
                                                        0      & 0      & \ldots \\
                                                    \end{pmatrix}             \\
        \label{eq_tgpisomorthend}   & =\begin{pmatrix}
                                           *                                                     & *      & * \\
                                           dDa^{-1}-M_{2n-2l}(d^{-1})^T(b^T)(a^{-1})^TK_lGa^{-1} & *      & * \\
                                           K_l(a^{-1})^TK_lG a^{-1}                              & \ldots & * \\
                                       \end{pmatrix}
    \end{align}
    From this, we find the matrix representing $\Ad(M_{a,b,c,d})$ as an endomorphism on $\Mat((2n-2l)\times l,\C)\times\mathbf{W}$, evaluated at $(D,G)$:
    \begin{equation}\label{eq_adsummary}
        \begin{pmatrix}
            (a^{-1})^T\otimes d & (a^{-1})^T\otimes-M_{2n-2l}(d^{-1})^T(b^T)(a^{-1})^TK_l \\
            0                   & (a^{-1})^T\otimes K_l(a^{-1})^TK_l                      \\
        \end{pmatrix},
    \end{equation}

    The determinant of $\Ad(M_{a,b,c,d})$ is then the product of the determinants of the diagonal blocks of (\ref{eq_adsummary}), acting as endomorphisms on $\Mat((2n-2l)\times l,\C)$ and $\mathbf{W}$ respectively.
    The upper-left block of (\ref{eq_adsummary}) has the determinant $\det((a^{-1})^T\otimes d)$.
    The lower-right block of (\ref{eq_adsummary}), acting on $\Mat(l\times l,\C)$, has the determinant $\det((a^{-1})^T\otimes K_l(a^{-1})^TK_l)=\det((a^{-1})^T)^{\otimes 2l}$. When this block is restricted to acting on $\mathbf{W}$, a complex vector space of dimension $l(l+1)/2$, we have the determinant:
    \begin{equation*}
        \det\left((a^{-1})^T\right)^{\otimes 2l\frac{l(l+1)/2}{l^2}}=\det\left((a^{-1})^T\right)^{\otimes(l+1)}.
    \end{equation*}

    Applying these results to our cocycles, we find that $(\det\circ\tau_{i,j})_{i,j\in I}$ is the product of cocycles of $\det(F^*\otimes(F^\perp/F))$ and $\det(F^*)^{\otimes(l+1)}$, so the isomorphism in (\ref{eq_tgpisomorth1}) follows.

    For the case of $F= F^\perp$, the calculations are the same, apart from removing the middle rows and columns in (\ref{eq_tgpisomorthstart},\ref{eq_tgpisomorthend}), so the isomorphism in (\ref{eq_tgpisomorth2}) follows.
\end{proof}

We can finally prove that Ramanathan-(semi)-stability for symplectic vector bundles is equivalent to checking slope-(semi)-stability conditions for isotropic subbundles.

\begin{theorem}\label{thm_ramversusslopesymp}
    For a symplectic vector bundle $(E,\beta)$ of rank $2n$, the following are equivalent:
    \begin{enumerate}[label=(\roman*)]
        \item $\Fr_\Sp(E,\beta)$ is Ramanathan-(semi)-stable.

        \item For all isotropic subbundles $F\subsetneq E$ of $(E,\beta)$, we have $\mu(F)(\leq)\mu(E)=0$.
    \end{enumerate}
\end{theorem}

\begin{proof}
    Let $F\subsetneq E$ be an isotropic subbundle of rank $l$. Using that $F^\perp/F$ is self-dual as a symplectic vector bundle, and is thus of degree $0$, the isomorphisms in Lemma \ref{lem_tgpisomsymp} imply:
    \begin{align*}
        \deg(s^*V_{\Fr_\Sp(E,\beta)/P_I})
         & =-\deg(F)(2n-2l) + \deg(F^\perp/F)l - \deg(F)(l+1) \\
         & =-\deg(F)(2n-l+1) +0l                              \\
         & =-\deg(F)(2n-l+1)
    \end{align*}
    Since $l \leq n$, we have $2n-l+1\geq n+1\geq 1$, so it follows that $\deg(s^*V_{\Fr_\Sp(E,\beta)/P_I})(\geq)0$ is equivalent to $\mu(F)(\leq)\mu(E)=0$.
\end{proof}

In the semistable case, the implication \textit{(i)} to \textit{(ii)} also follows from applying Theorem \ref{thm_instflag}, then Theorem \ref{thm_ramversusslope}, allowing us to also drop the isotropic condition.

\begin{lemma}\label{lem_ramversusslopesymp2}
    For a symplectic vector bundle $(E,\beta)$ of rank $2n$, the following are equivalent:
    \begin{enumerate}[label=(\roman*)]
        \item $\Fr_\Sp(E,\beta)$ is Ramanathan-semistable.

        \item $E$ is slope-semistable.
    \end{enumerate}
\end{lemma}

\begin{proof}
    \sloppy
    For \textit{(i)} to \textit{(ii)}, the inclusion $\Sp(2n,\C)\hookrightarrow\GL(2n,\C)$ maps $\Rad(\Sp(2n,\C))=0$ into $\Rad(\GL(2n,\C))$. This extends $\Fr_\Sp(E,\beta)$ as a principal $\Sp(2n,\C)$-bundle to $\Fr(E)$ as a principal $\GL(2n,\C)$-bundle, so $\Fr(E)$ is Ramanathan-semistable due to Theorem \ref{thm_instflag}. Then $E$ is slope-semistable following Theorem \ref{thm_ramversusslope}.

    The implication \textit{(ii)} to \textit{(i)} follows from Theorem \ref{thm_ramversusslopesymp}.
\end{proof}

\subsection{Special-orthogonal vector bundles}

Special-orthogonal vector bundles are defined similarly to symplectic vector bundles, where we replace the symplectic form with a nondegenerate symmetric form.

\begin{definition}\label{def_orthbund}
    Let $E$ be a holomorphic vector bundle of rank $r$ on a compact Riemann surface $X$, then:
    \begin{enumerate}[label=(\alph*)]
        \item An \textit{orthogonal vector bundle} $(E,\beta)$ is a vector bundle $E$ admitting a nondegenerate symmetric form $\beta: E \oplus  E \to X\times\C$.

        \item A \textit{special-orthogonal vector bundle} $(E,\beta,\tau)$ is an orthogonal vector bundle $(E,\beta)$ that is also special $(E,\tau)$, in the sense of Definition \ref{def_specialvecbun}.
    \end{enumerate}
\end{definition}

Much like for symplectic vector bundles, Ramanathan-(semi)-stability corresponds to checking slope conditions on isotropic subbundles, as seen in \cite[$\S$4]{ramanan_orthogonalandspinbundlesoverhyperellipticcurves}. To see this, we need to understand the root system of the \textit{special-orthogonal group} $\SO(r,\C)$, which we define with respect to the nondegenerate symmetric form $\beta(x,y):=\sum_{i=1}^r x_{r+1-i}y_i$:
\begin{align}\label{eq_orthform}
    \SO(r,\C):= \left\{A\in\SL(r,\C)\ \middle\vert\ A^TK_{r}A=K_{r}\right\}, &  & K_r :=\begin{pmatrix}
                                                                                            0 & 0                                  & 1 \\
                                                                                            0 & \rotatebox[origin=c]{35}{$\cdots$} & 0 \\
                                                                                            1 & 0                                  & 0
                                                                                        \end{pmatrix}\in\GL(r,\C).
\end{align}

Since $\SO(1,\C)$ is trivial and $\SO(2,\C)$ is abelian, we impose $r\geq3$ unless otherwise stated.
We intersect the standard Cartan and Borel subgroups of $\GL(r,\C)$ to obtain Cartan and Borel subgroups of $\SO(r,\C)$. The root system of $\SO(r,\C)$ depends on whether $r$ is even or odd, with the simple roots induced by our chosen Borel subgroup:
\begin{align*}
    \triangle & =\left\{\alpha_{l,l+1} \in \Lie(T)^*\ \middle\vert\ 1 \leq l \leq n-1\right\} \sqcup\left\{\alpha_n\in \Lie(T)^*\right\},           &  & r=2n+1\geq3, \\
    \triangle & =\left\{\alpha_{l,l+1}\in\Lie(T)^*\ \middle\vert\ 1 \leq l \leq n-1\right\} \sqcup\left\{\alpha_{n-1}+\alpha_n\in\Lie(T)^*\right\}, &  & r=2n\geq4,
\end{align*}
using notation from (\ref{eq_simpleroots},\ref{eq_simplerootssp}).
This difference in parity is significant, as it affects how standard parabolic subgroups appear. In the odd case, standard parabolic subgroups are all subgroups of upper-right block matrices. In the even case, \textit{not all} standard parabolic subgroups are upper-right block matrices, as seen concretely in the following example.

\begin{example}\label{ex_so4}
    The special-orthogonal group $\SO(4,\C)$ has the simple roots $\triangle=\{\alpha_{1,2},\alpha_1+\alpha_2\}$. The nontrivial (i.e., non-Borel and proper) standard parabolic subgroups are:
    \begin{align*}
        P_I & =\left\{\begin{pmatrix}
                          * & * & * & * \\
                          0 & * & 0 & * \\
                          * & * & * & * \\
                          0 & * & 0 & * \\
                      \end{pmatrix}\in\SO(4,\C)\right\}, & I & = \{\alpha_{1,2}\},      \\
        P_I & =\left\{\begin{pmatrix}
                          * & * & * & * \\
                          * & * & * & * \\
                          0 & 0 & * & * \\
                          0 & 0 & * & * \\
                      \end{pmatrix}\in\SO(4,\C)\right\}, & I & = \{\alpha_1+\alpha_2\}.
    \end{align*}
    Whilst only one of them consists of upper-right block matrices, the groups are isomorphic by an outer orthogonal automorphism: a conjugation permuting the middle rows and columns. Note that both groups are stabilizers of isotropic flags in $\C^4$ with one nontrivial subspace of dimension $2$.
\end{example}

Using the fact that parabolic subgroups always arise as stabilizers of flags, as seen in \cite[Theorem 3.9 (i)]{conrad_standardparabolicsubgroups}, we can  use \cite[Corollary 3.4]{conrad_standardparabolicsubgroups} to at least say that in the even case, parabolic subgroups of $\SO(r,\C)$ are conjugate, via an element $g \in \O(r,\C)$, to upper-right block matrix subgroups of $\SO(r,\C)$. Now we can proceed as in the symplectic case, being careful to respect odd and even cases. We define special-orthogonal frame bundles, similarly to Definitions \ref{def_specialframebun} and \ref{def_specialsympframebun}.

\begin{definition}\label{def_specialorthframebun}
    For a special-orthogonal vector bundle $(E,\beta,\tau)$, the \textit{special-orthogonal frame bundle} $\Fr_\SO(E,\beta,\tau)$ is a principal $\SO(r,\C)$-bundle, whose fiber at $x\in X$ is given by:
    \begin{equation*}
        \SOrth_{\C}^{\beta,\tau}(\C^{r},E_x),
    \end{equation*}
    consisting of isomorphisms $f:\C^{r}\rightarrow E_x$ compatible with $\beta$ on $E_x$, and the form on $\C^{r}$ induced by $K_{r}$ from (\ref{eq_orthform}), such that $\det(f)=\tau(x)$.
\end{definition}

We obtain a special-orthogonal variant of Lemmas \ref{lem_reductsubbundcorr}, \ref{lem_reductsubbundcorrsl}, and \ref{lem_reductsubbundcorrsymp}. Let $(E,\beta,\tau)$ be a special-orthogonal vector bundle of rank $r\geq3$.

\begin{lemma}\label{lem_reductsubbundcorrorth}
    The following are in correspondence:
    \begin{enumerate}[label=(\roman*)]
        \item Isotropic subbundles $F\subsetneq E$ of $(E,\beta,\tau)$ of rank $l$.

        \item Reductions $s^*\Fr_\SO(E,\beta,\tau)$ of $\Fr_\SO(E,\beta,\tau)$ to $P_I$, where in the odd $r=2n+1$ case:
              \begin{align}
                  \nonumber  I              & =\{\alpha_{l,l+1}\}\subseteq\triangle,          & l      & \neq n,     \\
                  \nonumber  I              & =\{\alpha_n\}\subseteq\triangle,                & l      & = n,        \\
                  \intertext{and where in the even $r=2n$ case:}
                  \nonumber  I              & =\{\alpha_{l,l+1}\}\subseteq\triangle,          & l      & \neq n-1,n, \\
                  \label{eq_orthevenissue}I & =\{\alpha_{n-1,n},\alpha_n+\alpha_{n+1}\}
                  \subseteq\triangle,       & l                                               & = n-1,               \\
                  \nonumber I               & =\{\alpha_{n-1}+\alpha_{n}\}\subseteq\triangle, & l      & = n.
              \end{align}
    \end{enumerate}
\end{lemma}

The proof is analogous to that of Lemma \ref{lem_reductsubbundcorrsymp}, where the distinction between odd and even cases is due to their differing root systems. Note that two simple roots are needed in (\ref{eq_orthevenissue}) when $l =n - 1$, since for example for $r=4$, a reduction to an isotropic subbundle of rank $l=1$ corresponds to a reduction to the Borel subgroup $B$ of $\SO(4,\C)$, because of Example \ref{ex_so4}.

Lemma \ref{lem_reductsubbundcorrorth} leads to a special-orthogonal version of Lemmas \ref{lem_tgpisom} and \ref{lem_tgpisomsymp}.

\begin{lemma}\label{lem_tgpisomorth}
    For an isotropic subbundle $F\subsetneq E$ of $(E,\beta,\tau)$ of rank $l$, and the corresponding reduction $s:X\rightarrow\Fr_\SO(E)/P_I$ from Lemma \ref{lem_reductsubbundcorrorth}, we have:
    \begin{align*}
        \det(s^*V_{\Fr_\SO(E,\beta,\tau)/P_I}) & \cong\det(F^*\otimes (F^\perp/F))\otimes\left(\det(F^*)^{\otimes(l-1)}\right), & F & \neq F^\perp, \\
        \det(s^*V_{\Fr_\SO(E,\beta,\tau)/P_I}) & \cong\det(F^*)^{\otimes(l-1)},                                                 & F & =F^\perp.
    \end{align*}
\end{lemma}

Again, the proof is analogous to that of Lemma \ref{lem_tgpisomsymp}. Note the slight difference to Lemma \ref{lem_tgpisomsymp} in that $\det(F^*)$ is tensored with itself $(l-1)$-times, not $(l+1)$-times. This is because the complex vector space $\mathbf{W}$ in the proof of Lemma \ref{lem_tgpisomsymp} is replaced with $\Lie(\SO(l,\C))$, which is of dimension $l(l-1)/2$, not $l(l+1)/2$. Finally, we have a correspondence between Ramanathan-(semi)-stability for special-orthogonal vector bundles, and slope conditions on isotropic subbundles, just like in Theorem \ref{thm_ramversusslopesymp}.

\begin{theorem}\label{thm_ramversusslopeorth}
    For a special-orthogonal vector bundle $(E,\beta,\tau)$ of rank $r\geq 3$, the following are equivalent:
    \begin{enumerate}[label=(\roman*)]
        \item $\Fr_\SO(E,\beta,\tau)$ is Ramanathan-(semi)-stable.

        \item For all isotropic subbundles $F\subsetneq E$ of $(E,\beta)$, we have $\mu(F)(\leq)\mu(E)=0$.

        \item The slope condition $\mu(F)(\leq)\mu(E)=0$ is fulfilled for all:
              \begin{enumerate}[label=(iii.\roman*)]
                  \item isotropic subbundles $F\subsetneq E$ of $(E,\beta,\tau)$, when $r=2n+1$ is odd.
                  \item isotropic subbundles $F\subsetneq E$ of $(E,\beta,\tau)$ of rank $l\neq n-1$, when $r=2n$ is even.
              \end{enumerate}
    \end{enumerate}
\end{theorem}

Unlike Theorem \ref{thm_ramversusslopesymp}, there is a third equivalent statement that only tests some isotropic subbundles, since not all isotropic subbundles correspond to standard maximal parabolic reductions, as seen in (\ref{eq_orthevenissue}) from Lemma \ref{lem_reductsubbundcorrorth}.

\begin{proof}
    The proof of equivalence between \textit{(i)} and \textit{(iii)} is analogous to the proof of Theorem \ref{thm_ramversusslopesymp}. Let $F\subsetneq E$ be an isotropic subbundle of rank $l$. Using that $F^\perp/F$ is self-dual as a special-orthogonal vector bundle, and is thus of degree $0$, the isomorphisms in Lemma \ref{lem_tgpisomorth} imply:
    \begin{align*}
        \deg(s^*V_{\Fr_\SO(E,\beta,\tau)/P_I})
         & =-\deg(F)(r-2l) + \deg(F^\perp/F)l - \deg(F)(l-1) \\
         & =-\deg(F)(r-l-1) +0l                              \\
         & =-\deg(F)(r-l-1)
    \end{align*}
    Since $l \leq r/2$ and $r\geq 3$, we have $r-l-1\geq 1$, so it follows that $\deg(s^*V_{\Fr_\Sp(E,\beta)/P_I})(\geq)0$ is equivalent to $\mu(F)(\leq)\mu(E)=0$.

    The implication \textit{(ii)} to \textit{(iii)} is clear, so it remains to show that \textit{(iii)} implies \textit{(ii)}. In the odd $r=2n+1$ case, there is nothing to show, so we focus on $r=2n$. We only have to show that for an isotropic subbundle $F\subsetneq E$ of rank $n-1$, we have $\mu(F)(\leq)\mu(E)=0$. Due to (\ref{eq_orthevenissue}) from Lemma \ref{lem_reductsubbundcorrorth}, $F$ corresponds to a reduction $s:X\rightarrow E/P_I$, where $I=\{\alpha_{n-1,n},\alpha_n+\alpha_{n+1}\}$. There exist cocycles $(\sigma_{i,j})_{i,j\in J}$ of $E$ that map into $P_I$, compatible with $F$ in the sense of Remark \ref{rem_subbundcocy}:
    \begin{equation*}
        \sigma_{i,j}:U_i\cap U_j\rightarrow P_I, \quad x\mapsto
        \begin{pmatrix}
            \alpha_{i,j}(x) & *              & *                   & *                        \\
            0               & \beta_{i,j}(x) & 0                   & *                        \\
            0               & 0              & \beta_{i,j}(x)^{-1} & *                        \\
            0               & 0              & 0                   & (\alpha_{i,j}(x)^{-1})^T \\
        \end{pmatrix},\quad i,j\in J,
    \end{equation*}
    where $(\alpha_{i,j})_{i,j\in J}$ are cocycles of $F$ and $\beta_{i,j}(x)\in \C^\times$. Since the cocycles $(\sigma_{i,j})_{i,j\in J}$ map into $P_{\{\alpha_n+\alpha_{n+1}\}}$, they define a reduction of $\Fr_{\SO}(E,\beta,\tau)$ to $P_{\{\alpha_n+\alpha_{n+1}\}}$, corresponding to a Lagrangian subbundle $F'$ of $E$ strictly containing $F$. At every fiber $E_x$, $x\in X$, the Lagrangian subspace $F'_x$ strictly containing $F_x$ exists following \cite[Theorem 3.9 (ii)]{conrad_standardparabolicsubgroups}. This is a quirk of the specific reductive group structure of $\SO(2n,\C)$, where parabolic subgroups are not in bijection with isotropic flags of $\C^{2n}$. Much like in Example \ref{ex_so4}, we can conjugate the cocycles $(\sigma_{i,j})_{i,j\in J}$ by an orthogonal matrix swapping the middle rows and columns, allowing us to swap $(\beta_{i,j})_{i,j\in J}$ with $(\beta_{i,j}^{-1})_{i,j\in J}$. In this way, we can assume $\deg(F'/F)\geq 0$ without loss of generality. Since $\deg(F)+\deg(F'/F)=\deg(F')(\leq)0$, and since $\deg(F'/F)\geq 0$, we have $\deg(F)(\leq)0$, and so $\mu(F)(\leq)\mu(E)=0$.
\end{proof}

In the semistable case, the implications \textit{(i)} to \textit{(ii)}, and \textit{(i)} to \textit{(iii)}, follow from Theorem \ref{thm_instflag}, which allows us to drop the isotropic condition, similarly to Lemma \ref{lem_ramversusslopesymp2}.

\begin{lemma}\label{lem_ramversusslopeorth2}
    For a special-orthogonal vector bundle $(E,\beta,\tau)$ of rank $r\geq 3$, the following are equivalent:
    \begin{enumerate}[label=(\roman*)]
        \item $\Fr_\SO(E,\beta,\tau)$ is Ramanathan-semistable.

        \item $E$ is slope-semistable.
    \end{enumerate}
\end{lemma}
The proof is analogous to that of Lemma \ref{lem_ramversusslopesymp2}.

\begin{remark}\label{rem_rank2}
    For a special-orthogonal vector bundle $(E,\beta,\tau)$ of rank $2$, $\Fr_\SO(E,\beta,\tau)$ is always Ramanathan-semistable, since $\SO(2,\C)$ is abelian.
\end{remark}

\section{Harder-Narasimhan filtrations}

\subsection{Harder-Narasimhan filtrations of decorated vector bundles}
\sloppy
We wish to construct \textit{Harder-Narasimhan filtrations} of symplectic and special-orthogonal vector bundles. Such filtrations of vector bundles were first constructed by Harder and Narasimhan in \cite[Lemma 1.3.8]{harder-narasimhan_onthecohomologygroupsofmodulispaces}.

\begin{theorem}\label{thm_hardernarasimhan}
    Let $E$ be a vector bundle of rank $r$ on a compact Riemann surface $X$. There exists a filtration of $E$ by subbundles:
    \begin{equation*}
        0=E_0\subsetneq\ldots\subsetneq E_t=E,
    \end{equation*}
    such that:
    \begin{enumerate}[label=(\roman*)]
        \item $F_m:=E_m/E_{m-1}$ is slope-semistable for all $m=1,\ldots,t$.

        \item $\mu(F_1)>\ldots>\mu(F_t)$.
    \end{enumerate}
    The filtration is unique amongst all filtrations of $E$ with these properties.
\end{theorem}

This filtration measures how far away the vector bundle $E$ is from being slope-semistable, as $E$ is slope-semistable if and only if its Harder-Narasimhan filtration is trivial. As we will see in the proof, each quotient $F_m$ contradicts the slope-semistability of a quotient bundle of $E$, so the Harder-Narasimhan filtration measures to what extent $E$ is slope-semistable. There exist analogous filtrations for symplectic and special-orthogonal vector bundles. As seen in \cite[$\S$10]{atiyah-bott_yangmills} and \cite{biswas-holla_harder-narasimhanreductionofprincbund}, these filtrations come from \textit{canonical reductions} of principal bundles. We sketch a proof of Theorem \ref{thm_hardernarasimhan} by constructing the filtration by hand.

\begin{proof} For our purposes, it is important to know just the main steps of the proof, so we refer to \cite[Lemma 1.3.8]{harder-narasimhan_onthecohomologygroupsofmodulispaces} for further details.

    \hspace{16.6em}\textbf{Step 1}

    Define the suprema:
    \begin{align*}
        \deg_{\max}(E) & :=\sup\left\{\deg(F)\ \middle\vert\ F\subseteq E \text{ subbundle}\right\}, \\
        \mu_{\max}(E)  & :=\sup\left\{\mu(F)\ \middle\vert\ F\subseteq E\text{ subbundle}\right\},
    \end{align*}
    which we claim are bounded. We first handle $\deg_{\max}(E)$. The $0$-th sheaf cohomology of a vector bundle is the complex vector space of global sections. For a subbundle $F$ of $E$ of rank $l=1,\ldots,r$, a global section of $F$ is also a global section of $E$, and thus $\dim_\C(H^0(X,E))\geq\dim_\C(H^0(X,F))$.
    We denote the genus of $X$ by $g$, we then have due to Riemann-Roch, from \cite[2. Theorem 2]{friedman_algebraicsurfacesandholomorphicvectorbundles}, that:
    \begin{align}
        \dim_{\C}(H^0(X,F))=\dim_{\C}(H^1(X,F))+\deg(F)+(1-g)l   & \geq\deg(F)+(1-g)l,\nonumber \\
        \dim_{\C}(H^0(X,E))+(g-1)l\geq\dim_{\C}(H^0(X,F))+(g-1)l & \geq\deg(F).\label{eq_rr2}
    \end{align}
    If $g=0$, we have $\dim_{\C}(H^0(X,E))\geq\deg(F)$ due to (\ref{eq_rr2}). Otherwise, $g\geq1$, and we can imply $\dim_{\C}(H^0(X,E))+(g-1)r\geq\deg(F)$ from (\ref{eq_rr2}), so $\deg_{\max}(E)$ is bounded. We now handle $\mu_{\max}(E)$. Let $l=1,\ldots,r$, then by using (\ref{eq_rr2}), we find a subbundle $F_l$ of $E$ of rank $l$, such that $F_l$ has maximal degree amongst all subbundles of $E$ of rank $l$.
    Due to the finite choices of $l$, we have $\mu_{\max}(E)=\max_{l=1}^r\mu(F_l)<\infty$.

    \hspace{16.6em}\textbf{Step 2}

    Since $\mu_{\max}(E)=\max_{l=1}^r\mu(F_l)<\infty$, there exists a subbundle $F\subseteq E$ of maximal rank such that $\mu(F)=\mu_{\max}(E)$.

    \hspace{16.6em}\textbf{Step 3}

    Verify that $F$ is the \textit{strongly contradicting semistability (SCSS)} subbundle $\cG(E)$ of $E$, as defined in \cite[1.3]{harder-narasimhan_onthecohomologygroupsofmodulispaces}, and is hence unique, as shown in \cite[Proposition 1.3.4]{harder-narasimhan_onthecohomologygroupsofmodulispaces}. Being SCCS means $\cG(E)$ is slope-semistable, and if $\cG(E)\subsetneq E$, then $\mu(\cG(E))>\mu_{\max}(E/\cG(E))$. One way to prove uniqueness is by showing that the projection $\cG(E)\rightarrow E/\cG(E)'$ to another SCSS bundle $\cG(E)'$ is trivial, i.e., equal to $0$, using that $\cG(E)$ is slope-semistable, and $\mu(\cG(E))>\mu_{\max}(E/\cG(E)')$.

    \hspace{16.6em}\textbf{Step 4}

    Recursively construct subbundles $E_m$ of $E$ such that $E_m/E_{m-1}=\cG(E/E_{m-1})$ until this process terminates at some $E_t=E$. We have constructed a Harder-Narasimhan filtration.

    \hspace{16.6em}\textbf{Step 5}

    For the uniqueness part, for another Harder-Narasimhan filtration:
    \begin{equation*}
        0=E_0'\subsetneq\ldots\subsetneq E_k'=E,
    \end{equation*}
    we perform an induction on the filtration length $t$. Given $t=1$, then if $k\geq 2$, the projection $E\rightarrow E/E'_{k-1}$ is trivial since $E$ is slope-semistable and $\mu(E)>\mu_{\max}(E/E'_{k-1})$, leading to a contradiction, so $k=1$. Assuming uniqueness for filtrations of length $t-1$, the filtration:
    \begin{equation*}
        0=E_1/E_1\subsetneq E_2/E_1\ldots\subsetneq E_t/E_1=E/E_1,
    \end{equation*}
    is the unique Harder-Narasimhan filtration of $E/E_1$. It remains to show $E_1'=E_1$, which follows from the uniqueness of SCSS bundles in \cite[Proposition 1.3.4]{harder-narasimhan_onthecohomologygroupsofmodulispaces}.
\end{proof}

Special vector bundles and their underlying vector bundles have the same slope-semistability conditions, coming from the Ramanathan-semistability of their special (frame) bundles, following Theorems \ref{thm_ramversusslope} and \ref{thm_ramversusslopespecial}, and Lemma \ref{lem_derived}. Thus, constructing Harder-Narasimhan filtrations for special vector bundles is identical to Theorem \ref{thm_hardernarasimhan}.

\begin{theorem}\label{thm_hardernarasimhanspecial}
    Let $(E,\tau)$ be a special vector bundle of rank $r$ on a compact Riemann surface $X$. There exists a filtration of $(E,\tau)$ by subbundles:
    \begin{equation*}
        0=E_0\subsetneq\ldots\subsetneq E_t=E,
    \end{equation*}
    such that:
    \begin{enumerate}[label=(\roman*)]
        \item $F_m:=E_m/E_{m-1}$ is slope-semistable for all $m=1,\ldots,t$.

        \item $\mu(F_1)>\ldots>\mu(F_t)$.
    \end{enumerate}
    The filtration is unique amongst all filtrations of $(E,\tau)$ with these properties.
\end{theorem}

When constructing Harder-Narasimhan filtrations for symplectic vector bundles, we must restrict our filtrations to isotropic subbundles, due to Theorem \ref{thm_ramversusslopesymp}.

\begin{theorem}\label{thm_hardernarasimhansymp}
    Let $(E,\beta)$ be a symplectic vector bundle of rank $r$ on a compact Riemann surface $X$. There exists a filtration of $(E,\beta)$ by isotropic subbundles:
    \begin{equation*}
        0=E_0\subsetneq\ldots\subsetneq E_t\subsetneq E,
    \end{equation*}
    such that:
    \begin{enumerate}[label=(\roman*)]
        \item $F_m:=E_m/E_{m-1}$ is slope-semistable for all $m=1,\ldots,t$.

        \item $\mu(F_1)>\ldots>\mu(F_t)>0$.

        \item $E_t^\perp/E_t$ is slope-semistable.
    \end{enumerate}
    The filtration is unique amongst all filtrations of $(E,\beta)$ with these properties.
\end{theorem}

\begin{proof} We follow the steps from the proof of Theorem \ref{thm_hardernarasimhan} and replace every instance of \textit{subbundle} with \textit{isotropic subbundle}.

    \hspace{16.6em}\textbf{Step 1}

    Define the suprema:
    \begin{align*}
        \deg_{\max}(E,\beta) & :=\sup\left\{\deg(F)\ \middle\vert\ F\subsetneq E \text{ isotropic subbundle}\right\}, \\
        \mu_{\max}(E,\beta)  & :=\sup\left\{\mu(F)\ \middle\vert\ F\subsetneq E\text{ isotropic subbundle}\right\}.
    \end{align*}
    From the proof of Theorem \ref{thm_hardernarasimhan}, we have $\deg_{\max}(E,\beta)\leq\deg_{\max}(E)<\infty$ and $\mu_{\max}(E,\beta)\leq\mu_{\max}(E)<\infty$.

    \hspace{16.6em}\textbf{Step 2}

    There exists an isotropic subbundle $F\subsetneq E$ of maximal rank such that $\mu(F)=\mu_{\max}(E,\beta)$. If $\mu(F)\leq0$, then the filtration:
    \begin{equation*}
        0=E_0\subsetneq E,
    \end{equation*}
    fulfills \textit{(i)} and \textit{(ii)} vacuously, and \textit{(iii)} is fulfilled due to Theorem \ref{thm_ramversusslopesymp} and Lemma \ref{lem_ramversusslopesymp2}, and the fact that $\mu(E)=0$. Otherwise, we assume $\mu(F)>0$.

    \hspace{16.6em}\textbf{Step 3}

    When $\mu(F)>0$, we verify that $F$ is a \textit{strongly contradicting semistability (SCSS) isotropic} subbundle $\cG(E,\beta)$ of $E$. Being SCCS isotropic means $\cG(E,\beta)$ is slope-semistable, $\mu(\cG(E,\beta))>0$, and if $\cG(E,\beta)$ is not maximally isotropic, then $\mu(\cG(E,\beta))>\mu_{\max}(\cG(E)^\perp/\cG(E),\beta)$.

    \hspace{16.6em}\textbf{Step 4}

    Recursively construct isotropic subbundles $E_m$ of $E$ such that $E_m/E_{m-1}=\cG(E_{m-1}^\perp/E_{m-1},\beta)$ until this process terminates when $E_t$ fulfills condition \textit{(iii)} for some $t$. We have constructed a symplectic Harder-Narasimhan filtration.

    \hspace{16.6em}\textbf{Step 5}

    For the uniqueness part, for another symplectic Harder-Narasimhan filtration:
    \begin{equation*}
        0=E_0'\subsetneq\ldots\subsetneq E_k'\subsetneq E,
    \end{equation*}
    we perform an induction on $t$. Given $t=0$, then due to condition (iii) and  Lemma \ref{lem_ramversusslopesymp2}, we have $k=0$. Assuming uniqueness for $t-1$, the filtration:
    \begin{equation*}
        0=E_1/E_1\subsetneq E_2/E_1\ldots\subsetneq E_t/E_1\subsetneq E_1^\perp/E_1,
    \end{equation*}
    is the unique symplectic Harder-Narasimhan filtration of $(E_1^\perp/E_1,\beta)$.
    It remains to show $E_1'=E_1$, which is done in multiple steps: First, the projection $E_1\rightarrow E/E_1'^\perp$ is trivial, i.e., equal to $0$, using that $E_1$ is slope-semistable, and $\mu(E_1)>0>\mu_{\max}(E/E_1'^\perp)=\mu(E/E_1'^\perp)$. The second inequality is true since $E/E_1'^\perp$ is isomorphic to the dual bundle $E_1'^*$ of $E_1'$, where $E_1'^*$ is slope-semistable, and where $\mu(E_1')>0$ implies $0>\mu(E_1'^*)$. Next, the projection $E_1\rightarrow E_1'^\perp/E_2'^\perp$ is trivial for similar reasons, since $E_1$ is slope-semistable, and $\mu(E_1)>0>\mu_{\max}(E_1'^\perp/E_2'^\perp)=\mu(E_1'^\perp/E_2'^\perp)$. Recursively, the projection $E_1\rightarrow E_{m-1}'^\perp/E_{m}'^\perp$ is trivial for all $m=1,\ldots,k$, and thus $E_1\subseteq E_t'^\perp$. If $E_t'$ is Lagrangian, then the projection $E_1\rightarrow E_t'^\perp/E_t'$ is trivial, since $E_t'^\perp/E_t'=0$. Otherwise, $E_t'$ is not Lagrangian, and to prove that $E_1\rightarrow E_t'^\perp/E_t'$ is trivial, we use that $E_1$ is slope-semistable, and $\mu(E_1)>0=\mu_{\max}(E_t'^\perp/E_t')=\mu(E_t'^\perp/E_t')$. The last equality follows from Theorem \ref{thm_ramversusslopesymp} and Lemma \ref{lem_ramversusslopesymp2}, and the fact that $\mu(E_t'^\perp/E_t')=0$.
    Due to $E_1\subseteq E_t'$, both $E_1$ and $E_1'$ are SCCS subbundles of $E_t'$, which are unique due to \cite[Proposition 1.3.4]{harder-narasimhan_onthecohomologygroupsofmodulispaces}. Hence, $E_1=E_1'$.
\end{proof}

The special-orthogonal Harder-Narasimhan filtration is constructed in an analogous way, but we have to make a slight distinction between the cases of even and odd ranks.

\begin{theorem}\label{thm_hardernarasimhanorthodd}
    Let $(E,\beta,\tau)$ be a special-orthogonal vector bundle of rank $r=2n+1\geq3$ on a compact Riemann surface $X$. There exists a filtration of $(E,\beta,\tau)$ by isotropic subbundles:
    \begin{equation*}
        0=E_0\subsetneq\ldots\subsetneq E_t\subsetneq E,
    \end{equation*}
    such that:
    \begin{enumerate}[label=(\roman*)]
        \item $F_m:=E_m/E_{m-1}$ is slope-semistable for all $m=1,\ldots,t$.

        \item $\mu(F_1)>\ldots>\mu(F_t)>0$.

        \item $E_t^\perp/E_t$ is slope-semistable.
    \end{enumerate}
    The filtration is unique amongst all filtrations of $(E,\beta,\tau)$ with these properties.
\end{theorem}

The proof of this theorem is analogous to that of Theorem \ref{thm_hardernarasimhanspecial}. Instead of applying Theorem \ref{thm_ramversusslopesymp} and Lemma \ref{lem_ramversusslopesymp2}, the proof makes use of Theorem \ref{thm_ramversusslopeorth} and Lemma \ref{lem_ramversusslopeorth2}. In the case of even rank, we need to modify condition \textit{(iii)} in special-orthogonal Harder-Narasimhan filtrations.

\begin{theorem}\label{thm_hardernarasimhanortheven}
    Let $(E,\beta,\tau)$ be a special-orthogonal vector bundle of rank $r=2n\geq4$ on a compact Riemann surface $X$. There exists a filtration of $(E,\beta,\tau)$ by isotropic subbundles:
    \begin{equation*}
        0=E_0\subsetneq\ldots\subsetneq E_t\subsetneq E,
    \end{equation*}
    such that:
    \begin{enumerate}[label=(\roman*)]
        \item $F_m:=E_m/E_{m-1}$ is slope-semistable for all $m=1,\ldots,t$.

        \item $\mu(F_1)>\ldots>\mu(F_t)>0$.

        \item $E_t$ is not of rank $n-1$, and $E_t^\perp/E_t$ is slope-semistable.
    \end{enumerate}
    The filtration is unique amongst all filtrations of $(E,\beta,\tau)$ with these properties.
\end{theorem}

The modification to condition \textit{(iii)} is needed since otherwise in some cases where $E_t$ is of rank $n-1$, we can extend the filtration to:
\begin{equation*}
    0=E_0\subsetneq\ldots\subsetneq E_t\subsetneq E_{t+1}\subsetneq E,
\end{equation*}
fulfilling \textit{(i)}, \textit{(ii)}, and \textit{(iii)}, which contradicts the uniqueness of Harder-Narasimhan filtrations. This extension was mentioned in the proof of Theorem \ref{thm_ramversusslopeorth}, where we constructed an isotropic subbundle $F'$ containing an isotropic subbundle $F=E_t$ of rank $n-1$, using \cite[Theorem 3.9 (ii)]{conrad_standardparabolicsubgroups} at every fiber $F_x$ of $F$. From the reductive groups we study, this is only possible in $\SO(2n,\C)$, where parabolic subgroups are not in bijection with isotropic flags of $\C^{2n}$.

Just like Theorem \ref{thm_hardernarasimhanortheven}, the proof is analogous to Theorem \ref{thm_hardernarasimhanspecial}, and makes use of Theorem \ref{thm_ramversusslopeorth} and Lemma \ref{lem_ramversusslopeorth2}.

\begin{remark}\label{rem_hncompare}
    By comparing all the variants of Harder-Narasimhan filtrations from Theorems \ref{thm_hardernarasimhan}-\ref{thm_hardernarasimhanortheven}, it is clear that the symplectic and special-orthogonal Harder-Narasimhan filtrations:
    \begin{equation*}
        0=E_0\subsetneq\ldots\subsetneq E_t\subsetneq E,
    \end{equation*}
    induce an extended filtration including the corresponding co-isotropic subbundles:
    \begin{equation*}
        0=E_0\subsetneq\ldots\subsetneq E_t\subseteq E_t^\perp\subsetneq\ldots\subsetneq E_0^\perp=E,
    \end{equation*}
    which is the Harder-Narasimhan filtration of the underlying vector bundle. When we introduce canonical reductions in Subsection \ref{subsec_ab}, we will see why this is expected.
\end{remark}

We want an example of special-orthogonal Harder-Narasimhan filtrations, and want to compare them to the usual Harder-Narasimhan filtrations. For this, we look at filtrations of adjoint bundles.

\begin{lemma}\label{lem_adjointspecialorth}
    For a principal $G$-bundle $\xi$, where $G$ is complex and reductive, there exists a special-orthogonal vector bundle structure on the adjoint bundle $\ad(\xi)$.
\end{lemma}

\begin{proof}
    Since $G$ is reductive, there exists a nondegenerate adjoint-invariant form on $\Lie(G)$. Using that $\ad(\xi)$ is induced by the extension of $\xi$ by the adjoint representation $\Ad:G\rightarrow\GL(\Lie(G))$, this form induces a nondegenerate symmetric form $\beta$ on $\ad(\xi)$, such that $(\ad(\xi),\beta)$ is an orthogonal vector bundle. To show that $(\ad(\xi),\beta)$ is special-orthogonal, we must show that the determinant bundle $\det(\ad(\xi))$ is isomorphic to the product bundle $X\times\C$. Since $G$ is an affine group, it can be viewed as a matrix subgroup of $\GL(r,\C)$, following \cite[I. 1.10 Proposition]{borel_linearalgebraicgroups}, so the adjoint representation $\Ad:G\rightarrow\GL(\Lie(G))$ can be viewed as matrix conjugation within $\GL(r,\C)$. Cocycles $(\sigma_{i,j})_{i,j\in J}$ of $\xi$ induce cocycles $(\Ad\circ\sigma_{i,j})_{i,j\in J}$ of $\ad(\xi)$, equal to $((\sigma_{i,j}^{-1})^T\otimes\sigma_{i,j})_{i,j\in J}$. This induces cocycles $(\det((\sigma_{i,j}^{-1})^T\otimes\sigma_{i,j}))_{i,j\in J}$ of $\det(\ad(\xi))$, which are cocycles of the product bundle $X\times\C$. Thus, the claim follows.
\end{proof}

Lemma \ref{lem_adjointspecialorth} does not give a unique special-orthogonal vector bundle structure on $\ad(\xi)$, as we made the choices of an adjoint-invariant form on $\Lie(G)$, cocycles $(\sigma_{i,j})_{i,j\in J}$ of $\xi$, and an embedding of $G$ into $\GL(r,\C)$. However, in concrete examples, e.g.\ $G = \Sp(2n,\C)$ or $G = \SO(r,\C)$, these are often canonical. For the upcoming example of special-orthogonal Harder-Narasimhan filtrations (Example \ref{ex_orthHNfiltration}), we will make use of the following result from \cite[Lemma 10.1]{atiyah-bott_yangmills}.

\begin{lemma}\label{lem_semistabletensor}
    Let $E$ and $E'$ be slope-semistable vector bundles of rank $r$ and $r'$, then $E\otimes E'$ is slope-semistable of rank $rr'$, with slope $\mu(E\otimes E')=\mu(E)+\mu(E')$.
\end{lemma}

\begin{proof}
    Due to properties of tensor products, $E\otimes E'$ has rank $rr'$.
    Using properties of degrees on tensor bundles, we obtain:
    \begin{equation*}
        \mu(E\otimes E')=\deg(E\otimes E')/rr'=(r'\deg(E)+r\deg(E'))/rr'=\mu(E)+\mu(E').
    \end{equation*}
    For the slope-semistability of $E\otimes E'$, we first assume that $E$ and $E'$ are slope-stable. In \cite[Appendix 4.1-4.2]{wells-garciaprada_differentialanalysis}, we see that subbundles $E$ and $E'$ correspond to irreducible unitary representations of an extension $\Gamma$ of $\pi_1(X)$. By tensoring these representations, we get another unitary representation of $\Gamma$ that corresponds to $E\otimes E'$, which is thus slope-semistable. Following \cite[Lemma 10.1]{atiyah-bott_yangmills}, to generalize to the situation where $E$ and $E'$ are only slope-semistable, we may use Jordan-Hölder filtrations of $E$ and $E'$, as discussed in \cite[1.5]{huybrechts-lehn_geometryofmodulispaces}.
\end{proof}

With Lemmas \ref{lem_adjointspecialorth} and \ref{lem_semistabletensor}, we can construct an example of a special-orthogonal vector bundle, and find its special-orthogonal Harder-Narasimhan filtration.

\begin{example}\label{ex_orthHNfiltration}
    Let $E$ and $E'$ be slope-semistable vector bundles, of ranks $r$ and $r'$ on a compact Riemann surface $X$. The direct sum $E\oplus E'$ has cocycles of the form:
    \begin{equation*}
        \sigma_{i,j}:U_i\cap U_j\rightarrow \GL(r,\C)\times\GL(r',\C),\quad x\mapsto(\alpha_{i,j}(x),\delta_{i,j}(x))
    \end{equation*}
    where $(\alpha_{i,j})_{i,j\in J}$ are cocycles of $E$ and $(\delta_{i,j})_{i,j\in J}$ are cocycles of $E'$. With respect to these cocycles, an embedding $\GL(r,\C)\times\GL(r',\C)\hookrightarrow\GL(r+r',\C)$ as diagonal block matrices,  and the adjoint-invariant form $\beta(X,Y)=\tr(XY)$ on $\Lie(\GL(r+r',\C))$, Lemma \ref{lem_adjointspecialorth} induces a special-orthogonal vector bundle structure $(\ad(E\oplus E'),\beta,\tau)$ on $\ad(E\oplus E'):=\ad(\Fr(E\oplus E'))$. Since $\ad(E\oplus E')$ is isomorphic to the endomorphism bundle on $E\oplus E'$, we have the isomorphisms:
    \begin{align}
        \nonumber\ad(E\oplus E') & \cong(E\oplus E')^*\otimes(E\oplus E')                                                                          \\
        \nonumber                & \cong(E^*\oplus E'^*)\otimes(E\oplus E')                                                                        \\
                                 & \cong(E^*\otimes E)\oplus(E'^*\otimes E)\oplus(E^*\otimes E')\oplus(E'^*\otimes E').\label{eq_orthHNfiltration}
    \end{align}
    With respect to the form $\beta(X,Y)=\tr(XY)$, we have the isotropic subbundles of $\ad(E\oplus E')$:
    \begin{align}
        E'^*\otimes E           & \subsetneq (E'^*\otimes E)^\perp=(E'^*\otimes E)\oplus(E^*\otimes E)\oplus(E'^*\otimes E'), \label{ex_orthHNfiltration2} \\
        \nonumber E^*\otimes E' & \subsetneq (E^*\otimes E')^\perp=(E^*\otimes E')\oplus(E^*\otimes E)\oplus(E'^*\otimes E'),
    \end{align}
    and the special-orthogonal subbundles $(E^*\otimes E)$, $(E'^*\otimes E')$, and $(E^*\otimes E)\oplus(E'^*\otimes E')$. Since $E$ and $E'$ are slope-semistable, each of the four summands from (\ref{eq_orthHNfiltration}) are slope-semistable due to Lemma \ref{lem_semistabletensor}, which helps us determine the Harder-Narasimhan filtrations of $\ad(E\oplus E')$. If $\mu(E)=\mu(E')$, then by using that slope-semistability is preserved under dualizing bundles, Lemma \ref{lem_semistabletensor} applied to (\ref{eq_orthHNfiltration}) implies that each summand of (\ref{eq_orthHNfiltration}) is slope-semistable. Furthermore, each summand has slope $0$, so it follows that $\ad(E\oplus E')$ is slope-semistable and thus has a trivial Harder-Narasimhan filtration. Also, $(\ad(E\oplus E'),\beta,\tau)$ has a trivial special-orthogonal Harder-Narasimhan filtration. Otherwise, let $\mu(E)>\mu(E')$. We have the following Harder-Narasimhan filtration of $E$ using (\ref{eq_orthHNfiltration}) and Lemma \ref{lem_semistabletensor}:
    \begin{equation*}
        0\subsetneq E'^*\otimes E\subsetneq (E'^*\otimes E)\oplus (E^*\otimes E)\oplus(E'^*\otimes E')\subsetneq\ad(E\oplus E').
    \end{equation*}
    However, the special-orthogonal Harder-Narasimhan filtration of $(\ad(E\oplus E'),\beta,\tau)$ is:
    \begin{equation}\label{ex_orthHNfiltration3}
        0\subsetneq E'^*\otimes E\subsetneq \ad(E\oplus E').
    \end{equation}
    As seen in Remark \ref{rem_hncompare}, the Harder-Narasimhan filtration is the extension of the special-orthogonal Harder-Narasimhan filtration, with the co-isotropic subbundles included.
\end{example}

\subsection{Canonical reductions in the sense of Atiyah and Bott}\label{subsec_ab}
We now make the connection between Harder-Narasimhan and canonical reductions in the works of Atiyah and Bott in \cite{atiyah-bott_yangmills} and Biswas and Holla in \cite{biswas-holla_harder-narasimhanreductionofprincbund}. In this subsection, we focus and recall the approach of Atiyah and Bott in \cite{atiyah-bott_yangmills}, for which we first need a technical result from \cite[Lemma 10.2]{atiyah-bott_yangmills}.

\begin{lemma}\label{lem_slopecomparisonlemma}
    Let $E$ and $E'$ be vector bundles with filtrations by subbundles:
    \begin{equation*}
        0=E_0\subsetneq\ldots\subsetneq E_t=E,\quad
        0=E'_0\subsetneq\ldots\subsetneq E'_{t'}=E',
    \end{equation*}
    both with slope-semistable quotients $F_m:=E_m/E_{m-1}$, $m=1,\ldots t$, and $F'_m:=E'_m/E'_{m-1}$, $m=1,\ldots t'$. Assuming the quotients fulfill $\mu(F_m)\geq q$, and $\mu(F'_m)<q$, for some $q$, every morphism $E\rightarrow E'$ is trivial, i.e., equal to $0$.
\end{lemma}

\begin{proof}
    If $t=t'=1$, then $E$ and $E'$ are slope-semistable, the claim uses that $E$ is semistable and $\mu(E)>\mu_{\max}(E')$. We perform a double induction on $(t,t')$. Assuming the claim is true for the pairs $(t,t')$, we wish to verify the claim for the lengths $(t+1, t')$, and $(t,t'+1)$. For $(t+1, t')$, the restriction of a morphism $E\rightarrow E'$ to $E_t\rightarrow E'$ is trivial due to the induction hypothesis. Similarly, $E/E_t\rightarrow E'$ is also trivial due to the induction hypothesis, thus $E\rightarrow E'$ is trivial. The argument for $(t,t'+1)$ is analogous.
\end{proof}

Let $\xi$ be a principal $G$-bundle on a compact Riemann surface $X$, where $G$ is complex and reductive. By making the required choices, we fix a special-orthogonal vector bundle structure $(\ad(\xi),\beta,\tau)$ through Lemma \ref{lem_adjointspecialorth}. We define canonical reductions of $\xi$ using the following lemma, using that $\ad(\xi)$ is a Lie algebra bundle with fibers isomorphic to $\Lie(G)$.

\begin{lemma}\label{lem_canonreduction}
    For the special-orthogonal Harder-Narasimhan filtration of $(\ad(\xi),\beta,\tau)$ from Theorems \ref{thm_hardernarasimhanorthodd} and \ref{thm_hardernarasimhanortheven}:
    \begin{equation*}
        0=\ad(\xi)_0\subsetneq\ldots\subsetneq \ad(\xi)_t\subsetneq \ad(\xi),
    \end{equation*}
    \begin{enumerate}[label=(\roman*)]
        \item $\ad(\xi)_m$ is a nilpotent Lie algebra subbundle for all $m=1,\ldots,t$.

        \item $\ad(\xi)_t^\perp$ is a parabolic Lie algebra subbundle.
    \end{enumerate}
\end{lemma}

\begin{proof}
    We first prove that $\ad(\xi)_t^\perp$ is a Lie algebra bundle. It suffices to show that the map $\varphi:\ad(\xi)_t^\perp\otimes\ad(\xi)_t^\perp\rightarrow\ad(\xi)/\ad(\xi)_t^\perp$ induced by the Lie bracket is trivial, i.e., equal to $0$. Using Lemma \ref{lem_semistabletensor}, there exists a filtration of $\ad(\xi)_t^\perp\otimes\ad(\xi)_t^\perp$, whose quotients are semistable of slope greater or equal to $0$. Furthermore, the slopes of the quotients of the Harder-Narasimhan filtration of $\ad(\xi)_t/\ad(\xi)_t^\perp$ are strictly lesser than $0$. By Lemma \ref{lem_slopecomparisonlemma}, $\varphi$ is trivial.

    For \textit{(i)}, Lemma \ref{lem_slopecomparisonlemma} similarly shows that for $m=1,\ldots,t$, the map $\varphi:\ad(\xi)_t\otimes\ad(\xi)_m\rightarrow\ad(\xi)/\ad(\xi)_{m-1}$ induced by the Lie bracket is trivial.

    For \textit{(ii)}, note that for all $x\in X$, we have $(\ad(\xi)_t^\perp)_x\cong(\ad(\xi)_t^\perp/\ad(\xi)_t)_x\oplus(\ad(\xi)_t)_x$ as Lie algebras, where $(\ad(\xi)_t^\perp/\ad(\xi)_t)_x$ is a reductive Lie algebra. Due to \textit{(i)}, $(\ad(\xi)_t)_x$ is a nilpotent Lie algebra, and thus it is isomorphic to the nilpotent radical of $(\ad(\xi)_t^\perp)_x$.
    For all $x\in X$, we have found the Levi decomposition of $(\ad(\xi)_t^\perp)_x$, hence it is a parabolic Lie subalgebra of $\ad(\xi)_x$. Thus, the claim of \textit{(ii)} follows.
\end{proof}

Since $\ad(\xi)_t^\perp$ is a parabolic Lie algebra subbundle, we can state the definition of a canonical reduction of a principal bundle, following Atiyah and Bott in \cite[Page 589]{atiyah-bott_yangmills}.

\begin{definition}\label{can_red_def}
    A reduction $s^*\xi$ of $\xi$ to a parabolic subgroup $P$ of $G$ is called a \textit{canonical reduction} if $\ad(s^*\xi)$ is isomorphic to $\ad(\xi)_t^\perp$, where $\ad(\xi)_t$ is the largest isotropic subbundle in a special-orthogonal Harder-Narasimhan filtration of $\ad(\xi)$.
\end{definition}

The filtration in this definition exists due to Lemma \ref{lem_adjointspecialorth} and Theorems \ref{thm_hardernarasimhanorthodd} and \ref{thm_hardernarasimhanortheven}. Let us revisit Example \ref{ex_orthHNfiltration} to see how canonical reductions are related to Harder-Narasimhan filtrations.

\begin{example}\label{ex_canonicalreduction}
    In the setting of Example \ref{ex_orthHNfiltration}, with slope-semistable vector bundles $E$ and $E'$ of ranks $r$ and $r'$, such that $\mu(E)>\mu(E')$, we know that $E\oplus E'$ has the Harder-Narasimhan filtration:
    \begin{equation*}
        0\subsetneq E\subsetneq E\oplus E'.
    \end{equation*}
    Through Lemma \ref{lem_reductsubbundcorr}, this filtration defines a reduction $s$ of the frame bundle $\Fr(E\oplus E')$ to a standard maximal parabolic subgroup $P_{\{\alpha_{r,r+1}\}}$ of $\GL(r+r',\C)$. The adjoint bundle $\ad(s^*\Fr(E\oplus E'))$ is isomorphic to the endomorphisms of $E\oplus E'$ that preserve the subbundle $E$, hence, $\ad(s^*\Fr(E\oplus E'))$ is isomorphic to $(E'^*\oplus E)^\perp$ from (\ref{ex_orthHNfiltration2}). Due to the special-orthogonal Harder-Narasimhan filtration from (\ref{ex_orthHNfiltration3}), the reduction $s$ is a canonical reduction of the frame bundle $\Fr(E\oplus E')$.
\end{example}

This is an example of how Harder-Narasimhan filtrations of vector bundles induce canonical reductions of frame bundles. In \cite[Pages 589-590]{atiyah-bott_yangmills}, it is explained that Harder-Narasimhan filtrations of vector bundles always arise from canonical reductions of frame bundles.

It is fair to ask whether canonical reductions always exist in the general setting of principal $G$-bundles, where $G$ is complex and reductive.

\begin{lemma}\label{lem_canonicalreductionexist}
    For a principal $G$-bundle $\xi$, a canonical reduction $s^*\xi$ to a parabolic subgroup $P$ of $G$ exists.
\end{lemma}

\begin{proof}
    Since $\ad(\xi)_t^\perp$ is a parabolic Lie algebra subbundle of $\ad(\xi)$, due to Lemma \ref{lem_canonreduction}, its fibers are isomorphic to a parabolic Lie subalgebra of $\Lie(G)$. Through the exponential map, the image of this subalgebra generates a parabolic subgroup $P$ of $G$, such that the fibers of $\ad(\xi)_t^\perp$ are isomorphic to $\Lie(P)$. We fix an embedding $\iota:G\hookrightarrow\GL(r,\C)$ from \cite[I. 1.10 Proposition]{borel_linearalgebraicgroups}, which induces an embedding $\GL(\Lie(G))\hookrightarrow\GL(\Lie(\GL(r,\C)))\cong\GL(\C^{r^2})$, to view $\ad(\xi)$ as a vector bundle of rank $r^2$.
    By taking the frame bundle of $\ad(\xi)$, Lemma \ref{lem_reductsubbundcorr} recasts the subbundle $\ad(\xi)_t^\perp$ as a section $s':X\rightarrow\Fr(\ad(\xi))/P'$, where $P'$ is the standard maximal parabolic subgroup of $\GL(\C^{r^2})$ that stabilizes $D\iota(\Lie(P))\subseteq\C^{r^2}$, where $D\iota$ is the derivative of $\iota$. For every $x\in X$, the corresponding fiber $\Fr(\ad(\xi))_x/P'$ is a homogeneous space isomorphic to $\GL(\C^{r^2})/P'$, we fix such an isomorphism (which is not unique). Every element of $\GL(\C^{r^2})/P'$ corresponds to a linear map $D\iota(\Lie(P))\rightarrow\C^{r^2}/D\iota(\Lie(P))$. Since $s'$ is a reduction of $\Fr(\ad(\xi))$, it maps $x\in X$ to an element corresponding to a linear map $D\iota(\Lie(P))\rightarrow D\iota(\Lie(G))/D\iota(\Lie(P))$. Through the exponential map, the image of this linear map is an element of $G/P$. A choice of isomorphism $\Fr(\ad(\xi))_x/P'\cong\GL(\C^{r^2})/P'$ also determines an isomorphism $\xi_x/P\cong G/P$ through $\iota$, using that the kernel of $\Ad:G\rightarrow\GL(\Lie(G))$ lies inside $P$. Since the construction of an element in $G/P$ above is compatible with local trivializations and gluing, we have found a reduction $s:X\rightarrow\xi/P$. This reduction is canonical by construction.
\end{proof}

The next natural question to ask is whether canonical reductions are unique. The choices made in giving $\ad(\xi)$ a special-orthogonal vector bundle structure in Lemma \ref{lem_adjointspecialorth}, which also affect the construction of a canonical reduction in Lemma \ref{lem_canonicalreductionexist}, could lead to different canonical reductions. For this, the following lemma from \cite[Proposition 10.4]{atiyah-bott_yangmills} is useful.

\begin{lemma}
    For a canonical reduction $s^*\xi$ of a principal $G$-bundle $\xi$, and a homomorphism of complex reductive groups $G\rightarrow G'$ that maps $\Rad(G)$ into $\Rad(G')$, then the extension $(s^*\xi)(G')$ is a canonical reduction of $\xi(G')$.
\end{lemma}

This lemma explains the phenomenon we saw in Remark \ref{rem_hncompare}, but also that with the appropriate extensions of structure groups, the adjoint bundles $\ad(\xi)^\perp_t$ of canonical reductions are unique up to isomorphism. The uniqueness of canonical reductions then hinges upon how two reductions are related when their adjoint bundles are isomorphic, we claim that these reductions are \textit{conjugate} to each other, a notion we now make precise. We recall the following result from \cite[Lemma 2.2]{biswas-holla_harder-narasimhanreductionofprincbund}.

\begin{lemma}\label{lem_adequivalentreduction}
    Let $P$ and $P'$ be parabolic subgroups of $G$, let $s^*\xi$ be a reduction of $\xi$ to $P$, and let $s'^*\xi$ be a reduction of $\xi$ to $P'$.
    If there exists an isomorphism $\varphi:\ad(s^*\xi)\rightarrow\ad(s'^*\xi)$, then $P$ and $P'$ are conjugate, so there exists $g\in G$, such that $\conj_g:G\rightarrow G$, $h\mapsto ghg^{-1}$ fulfills $\conj_g(P)=P'$.

    Through $\conj_g$, we induce an extension $(s^*\xi)(P')$ of $s^*\xi$ to $P'$, and there exists an isomorphism $\psi:(s^*\xi)(P')\rightarrow(s'^*\xi)$ of principal $P'$-bundles.
\end{lemma}

The two reductions $s^*\xi$ and $s'^*\xi$ are said to be \textit{conjugate} to each other.

\begin{proof}
    Given $\varphi:\ad(s^*\xi)\rightarrow\ad(s'^*\xi)$, the fibers of the adjoint bundles are isomorphic, implying that the Lie algebras $\Lie(P)$ and $\Lie(P')$ are isomorphic. Thus, as parabolic subgroups of $G$, $P$ and $P'$ are conjugate to each other. Let $g\in G$ such that $\conj_g(P)=P'$. For all $x\in X$, $\varphi$ restricts to an isomorphism on the fibers at $x$:
    \begin{align*}
        \ad(s^*\xi)_x= P\backslash(s(x)\times\Lie(P)) & \rightarrow \ad(s'^*\xi)_x=P'\backslash(s'(x)\times\Lie(P')), \\
        [v,Y]                                         & \mapsto [\eta_x(v),\Ad(g)(Y)],
    \end{align*}
    through which we induce a map $\eta_x:s(x)\rightarrow s'(x)$ that induces an isomorphism $\eta:s^*\xi\rightarrow s'^*\xi$ of fiber bundles. To ensure that $\eta$ induces an isomorphism $\psi:(s^*\xi)(P')\rightarrow(s'^*\xi)$ of principal bundles, we need to verify that for all $x\in X$, $\eta_x$ is $P'$-equivariant with respect to the conjugation $\conj_g$. For all $v\in s(x)$ and all $p\in P$, we have:
    \begin{align*}
        [\eta_x(vp),\Ad(g)(Y)] & =[\eta_x(v),\Ad(gp^{-1}g^{-1}g)(Y)] \\
                               & =[\eta_x(v)gpg^{-1},\Ad(g)(Y)].
    \end{align*}
    Thus, the claim follows.
\end{proof}

Due to Lemma \ref{lem_adequivalentreduction}, two canonical reductions $s^*\xi$ and $s'^*\xi$ of $\xi$ to $P$ and $P'$ are conjugate in the sense that there exists an isomorphism $\psi:(s^*\xi)(P')\rightarrow(s'^*\xi)$ of principal $P'$-bundles. To recap, Harder-Narasimhan filtrations of vector bundles, as well as symplectic and special-orthogonal variants, generalize nicely to canonical reductions of principal $G$-bundles, where $G$ is complex and reductive.

\subsection{Canonical reductions in the sense of Biswas and Holla}\label{subsec_bh}

The approach of Biswas and Holla in \cite{biswas-holla_harder-narasimhanreductionofprincbund} to constructing canonical reductions is slightly different, using the notion of \textit{dominant characters} of parabolic subgroups $P$ of $G$, as discussed first by Ramanathan in \cite{ramanathan_stableprincipalbundles}. In that setting, a canonical reduction $s^*\xi$ of $\xi$ to a parabolic subgroup $P$ of $G$ has two key properties:
\begin{enumerate}[label=(BH-\roman*)]
    \item\label{bh1} The extension of $s^*\xi$ to the Levi-factor $L$ of $P$ is Ramanathan-semistable.

    \item\label{bh2} The degree of the extension of $s^*\xi$ with respect to dominant characters $P$ of $G$ is larger than $0$.
\end{enumerate}
More precisely, Biswas and Holla prove the following theorem, for which we will sketch a proof.
\begin{theorem}\label{thm_bh}
    A reduction $s^*\xi$ of $\xi$ is a canonical reduction if and only if it fulfills conditions \ref{bh1} and \ref{bh2}.
\end{theorem}

The first condition generalizes quotients in the Harder-Narasimhan filtration being slope-semistable, whilst the second condition generalizes the decreasing slopes of the quotients in the Harder-Narasimhan filtration. To understand \ref{bh2}, we define dominant characters.
\begin{definition}\label{def_domachar}
    Let $P$ be a parabolic subgroup of $G$. A nontrivial character $\chi:P\rightarrow\C^\times$ is a \textit{dominant character} if there exists a Cartan $T$ and Borel subgroup $B$, with $T\subseteq B\subseteq G$, inducing simple roots $\triangle$, such that the derivative $D\chi|_{\Lie(T)}:\Lie(T)\rightarrow\C$ is a nonnegative integer linear combination of simple roots in $\triangle$.
\end{definition}

\begin{example}\label{ex_character4}
    Let $P_I$ be the standard parabolic subgroup of $\GL(4,\C)$ corresponding to $I=\{\alpha_{2,3}\}\subset\triangle$, using the notation of (\ref{eq_simpleroots}). The character:
    \begin{equation*}
        \chi_2:P_I\rightarrow\C^\times,\quad\begin{pmatrix}
            A & B \\
            0 & D \\
        \end{pmatrix}\mapsto\det(A)/\det(D),
    \end{equation*}
    is a dominant character, since $D\chi_2|_{\Lie(T)}=\alpha_{1,2}+2\alpha_{2,3}+\alpha_{3,4}$.
\end{example}

Biswas and Holla's proof approach to Theorem \ref{thm_bh} is as follows. First, they verify the existence of reductions fulfilling \ref{bh1} and \ref{bh2}, and prove that they are unique up to conjugation. Then, by comparing \ref{bh1} and \ref{bh2} to the conditions for canonical reductions, Biswas and Holla find them equivalent. We now verify that in the vector bundle case, where $G=\GL(r,\C)$, the conditions \ref{bh1} and \ref{bh2} are equivalent to canonical reductions, providing an idea of how to solve the general case.

\begin{lemma}\label{lem_hnreduct}
    Let $E$ be a vector bundle of rank $r$. A reduction $s^*\Fr(E)$ of $\Fr(E)$ is a canonical reduction in the sense of Definition \ref{can_red_def} if and only if it fulfills conditions \ref{bh1} and \ref{bh2}.
\end{lemma}

\begin{proof}
    Due to Lemma \ref{lem_adequivalentreduction}, and the fact that conditions \ref{bh1} and \ref{bh2} are invariant under conjugation of $P$, we can assume without loss of generality that $P=P_I$ is a standard parabolic subgroup of $\GL(r,\C)$. From \cite[Pages 589-590]{atiyah-bott_yangmills}, we know that Harder-Narasimhan reductions of $E$ arise from canonical reductions $s^*\Fr(E)$ of $\Fr(E)$ to $P_I$. More precisely, by generalizing Lemma \ref{lem_reductsubbundcorr} from subbundles of $E$ to filtrations of $E$, $s^*\Fr(E)$ corresponds to a filtration of $E$, which is the Harder-Narasimhan filtration of $E$. To show \ref{bh1}, we use that the extension $(s^*\Fr(E))(L_I)$ to the Levi-factor $L_I$ of $P_I$ is the frame bundle of the product of the successive quotients $F_m$, $m=1,\ldots,t$ of the Harder-Narasimhan filtration of the vector bundle $E$ (the graded vector bundle associated to that filtration). Since the quotients are slope-semistable, their frame bundles are Ramanathan-semistable due to Theorem \ref{thm_ramversusslope}. Hence, $(s^*\Fr(E))(L_I)$ is Ramanathan-semistable as a product of Ramanathan-semistable bundles. Now we prove \ref{bh2}. We recall that characters $\chi:P_I\rightarrow\C^\times$ are completely determined by their restriction $\chi|_{L_I}:L_I\rightarrow\C^\times$ to the Levi-factor, as in the Levi decomposition $P_I=U_IL_I$, $U_I$ is a unipotent group (and thus has trivial characters). Since $L_I$ is the product of groups $\GL(r_m,\C)$, $m=1\ldots,t$, the dominant characters $\chi:P_I\rightarrow\C^\times$ restricted to $L_I$ are easy to enumerate: They are generated by the characters $\chi_m$, $m=1\ldots,t-1$, that evaluate the determinant of the $m$-th diagonal block, and divide by the determinant of the $m+1$-st diagonal block, as seen in Example \ref{ex_character4}.
    Hence, it suffices to show the inequality in \ref{bh2} for characters $\chi_m$, $m=1\ldots,t-1$. We have:
    \begin{equation*}
        \deg(\chi_m(s^*\Fr(E)))=\deg(F_m\otimes F_{m+1}^*)=\deg(F_m)\rk(F_{m+1})-\deg(F_{m+1})\rk(F_m)>0,
    \end{equation*}
    due to the slope inequalities of the Harder-Narasimhan filtration of $E$.
\end{proof}

To generalize Lemma \ref{lem_hnreduct} to the general principal $G$-bundle case, where $G$ is complex reductive, and prove Theorem \ref{thm_bh}, we need to review some facts about characters. For the rest of this subsection, we fix a Cartan $T$ and a Borel $B$ of $G$, such that $T\subseteq B\subseteq G$ with the simple roots $\triangle$.

\begin{lemma}\label{lem_canonreduct}
    For a principal $G$-bundle $\xi$, and a standard parabolic subgroup $P_I$ of $G$:
    \begin{enumerate}[label=(\roman*)]
        \item\label{lem_canonreducti} For all $\alpha\in I$, there exists a nontrivial character $\chi_\alpha:P_I\rightarrow\C^\times$, such that $D\chi_\alpha|_{\Lie(T)}:\Lie(T)\rightarrow\C$ is an integer linear combination of elements in $\triangle$, whose coefficient in $\alpha$ is positive, and whose coefficients in $I\setminus\{\alpha\}$ are $0$.

        \item\label{lem_canonreductii} For a reduction $s^*\xi$ of $\xi$ to a standard parabolic subgroup $P_I$ of $G$, fulfilling \ref{bh1}, condition \ref{bh2} is fulfilled if and only if for all $\alpha\in I$, there exists a character $\chi_\alpha$ from \ref{lem_canonreducti} such that $\deg(\chi_\alpha(s^*\xi))>0$.
    \end{enumerate}
\end{lemma}

For a proof of this lemma, see the middle section of the proof of \cite[Proposition 3.1]{biswas-holla_harder-narasimhanreductionofprincbund}. In the $\GL(r,\C)$ case, we can think of $\chi_\alpha$, $\alpha\in I$ as corresponding to $\chi_m:P_I\rightarrow\C^\times$ from the proof of Lemma \ref{lem_hnreduct} and from Example \ref{ex_character4}. The above lemma tells us that it suffices to check the inequality of \ref{bh2} on finitely many characters, as we did in the proof of Lemma \ref{lem_hnreduct}. Another part of \cite[Proposition 3.1]{biswas-holla_harder-narasimhanreductionofprincbund} generalizes \textbf{Step 1} in our Harder-Narasimhan filtration proofs from Theorems \ref{thm_hardernarasimhan}-\ref{thm_hardernarasimhanortheven}.

\begin{lemma}\label{lem_degmin}
    For a principal $G$-bundle $\xi$:
    \begin{enumerate}[label=(\roman*)]
        \item Let:
              \begin{equation*}
                  \deg_{\max}(\xi):=\sup\left\{\deg(\ad(s^*\xi))\ \middle\vert\ \begin{array}{l}P_I\text{ is a standard parabolic subgroup of }G, \\
                      s^*\xi\text{ is a reduction of }\xi\text{ to }P_I
                  \end{array}\right\}.
              \end{equation*}
              We have $\deg_{\max}(\xi)<\infty$, and there exists reductions $s^*\xi$ of $\xi$ to $P_I$, such that $\deg(\ad(s^*\xi))=\deg_{\max}(\xi)$.

        \item Let:
              \begin{equation}\label{eq_degmin2}
                  \cP:=\left\{P_I\text{ is a standard parabolic subgroup of }G\ \middle\vert\ \begin{array}{l} \exists s^*\xi\text{ reduction of }\xi\text{ to }P_I, \\
                      \deg(\ad(s^*\xi))=\deg_{\max}(\xi)
                  \end{array}\right\}.
              \end{equation}
              There exists an element $P_I$ of $\cP$ that is maximal in terms of inclusion.
    \end{enumerate}
\end{lemma}

\begin{proof}
    We first prove \textit{(i)}. Due to the short exact sequence from Remark \ref{rem_adbundleexactseq}, for all reductions $s^*\xi$ of $\xi$ to a standard parabolic subgroup $P_I$ of $G$, $\ad(s^*\xi)$ is isomorphic to a subbundle of $\ad(\xi)$.
    Just like in \textbf{Step 1} of the proof of Theorem \ref{thm_hardernarasimhan}, we know that the degrees of subbundles of $\ad(\xi)$ are bounded from above, i.e., $\deg_{\max}(\ad(\xi))<\infty$. Thus, \textit{(i)} follows.

    Since $\cP$ is nonempty, we can clearly find a maximal element in $\cP$, proving \textit{(ii)}.
\end{proof}

Lemma \ref{lem_degmin} determines a reduction $s^*\xi$ of $\xi$, which we claim fulfills conditions \ref{bh1} and \ref{bh2}. The following lemma is taken from \cite[Proposition 3.1]{biswas-holla_harder-narasimhanreductionofprincbund}.

\begin{lemma}\label{thm_canonreduct2}
    There exists a reduction $s^*\xi$ of $\xi$ to a standard parabolic subgroup $P_I$ of $G$, such that:
    \begin{enumerate}[label=(\roman*)]
        \item\label{thm_canonreduct2i} The degree $\deg(\ad(s^*\xi))$ is equal to $\deg_{\max}(\xi)$.

        \item\label{thm_canonreduct2ii} The parabolic subgroup $P_I$ is maximal, in terms of inclusion, within $\cP$ from (\ref{eq_degmin2}).
    \end{enumerate}
    Furthermore, $s^*\xi$ is a reduction of $\xi$ to $P_I$ fulfilling \ref{bh1} and \ref{bh2}.
\end{lemma}

The existence of a reduction fulfilling \ref{thm_canonreduct2i} and \ref{thm_canonreduct2ii} follows from Lemma \ref{lem_degmin}. Proving \ref{bh1} is done by first assuming $(s^*\xi)(L_I)$ is not Ramanathan-semistable, and reaching a contradiction to $\deg(\ad(s^*\xi))=\deg_{\max}(\xi)$. For \ref{bh2}, we use Lemma \ref{lem_canonreduct} so that it suffices to show $\deg(\chi_\alpha(s^*\xi))>0$ for certain characters $\chi_\alpha$. For proof details, see \cite[Proposition 3.1]{biswas-holla_harder-narasimhanreductionofprincbund}. Now that we have the existence of reductions fulfilling \ref{bh1} and \ref{bh2}, we know that they are unique up to conjugation. The following theorem is taken from \cite[Theorem 4.1]{biswas-holla_harder-narasimhanreductionofprincbund}.

\begin{theorem}\label{thm_canonreductunique}
    For reductions $s^*\xi$ of $\xi$ to $P_I$ and $s'^*\xi$ of $\xi$ to $P_{I'}$, fulfilling \ref{bh1} and \ref{bh2}, there exists a conjugation $\conj_g$, for $g\in G$, such that $\conj_g(P)=P'$, and an isomorphism $\psi:(s^*\xi)(P')\rightarrow s'^*\xi$.
\end{theorem}

Finally, we can prove Theorem \ref{thm_bh}, showing that canonical reductions in the sense of Definition \ref{can_red_def} are equivalent to fulfilling conditions \ref{bh1} and \ref{bh2}.

\begin{proof}[Proof of Theorem \ref{thm_bh}]
    It suffices to show that a canonical reduction fulfills conditions \ref{bh1} and \ref{bh2}. The other implication follows from the existence and uniqueness of canonical reductions, and the existence and uniqueness of reductions fulfilling \ref{bh1} and \ref{bh2} from Lemma \ref{thm_canonreduct2} and Theorem \ref{thm_canonreductunique}. Due to Lemma \ref{thm_canonreduct2}, it also suffices to show that for a canonical reduction $s^*\xi$ of $\xi$ to $P_I$, we have $\deg(\ad(s^*\xi))=\deg_{\max}(\xi)$, and that
    the parabolic subgroup $P_I$ is maximal, in terms of inclusion, within $\cP$ from (\ref{eq_degmin2}). Both claims follow from the fact that $\ad(s^*\xi)$ is isomorphic to $\ad(\xi)_t^\perp$ in the special-orthogonal Harder-Narasimhan filtration from Lemma \ref{lem_canonreduction}.
\end{proof}

\begin{remark}\label{rem_behrend}
    Biswas and Holla provided a bundle-theoretic proof to the existence of canonical reductions of principal bundles. Earlier in
    \cite{behrend_semi-stabilityofreductivegroupschemesovercurves}, Behrend provides a proof of the same result using his theory of \textit{complementary polyhedra}. From a principal $G$-bundle $\xi$, he constructs its complementary polyhedron, a root-theoretic object with a notion of semistability equivalent to Ramanathan-semistability: condition \ref{bh1} corresponds to verifying the stability of a projection of the complementary polyhedron, and condition \ref{bh2} is replaced with checking the positivity of the \textit{numerical invariants} of the complementary polyhedron. Behrend's approach to canonical reductions is particularly useful when generalizing canonical reductions to other interesting objects, such as Higgs bundles and parahoric torsors.
\end{remark}

\section{Harder-Narasimhan types}\label{sec_hntypes}

In this section, we learn how canonical reductions can be classified by obstruction classes and Harder-Narasimhan types, following \cite{friedman-morgan_ontheconversetoatheoremofatiyahbott} and \cite{ho-liu_yangmillsconnectionsorientable}. We then determine Harder-Narasimhan types for principal bundles with structure groups $\GL(r,\C)$, $\SO(r,\C)$, and $\Sp(2n,\C)$. To do this, we first review real forms $K$ of complex reductive groups $G$, which are maximal compact subgroups of $G$. Then, we review the algebraic characterization of the fundamental groups of $T$, $G$, $[G,G]$ and $G_{\ab}=G/[G,G]$ as $\mathbb{Z}$-modules.
Using Čech cohomology, we are then able to define obstruction classes and topological types of principal bundles.

\subsection{Fundamental groups of reductive groups}\label{subsec_fundamental}

We follow \cite[§26.1]{fulton-harris_representationtheory} to investigate real forms of $G$ and $\Lie(G)$.

\begin{definition}\label{def_realform}
    \begin{enumerate}[label=(\alph*)]
        \item A \textit{real form} $\k$ of $\Lie(G)$ is a real Lie subalgebra, whose complexification $\k\otimes_{\R}\C$ is isomorphic to $\Lie(G)$.

        \item Let $\k$ be a real form of $\Lie(G)$, such that its Killing-form $\kappa_{\k}$, a restriction of the Killing-form $\kappa$ of $\Lie(G)$, is negative-semidefinite. We call $\k$ a \textit{compact real form} of $\Lie(G)$.

        \item Let $K$ be a maximal compact real Lie subgroup of $G$, then $K$ is a \textit{compact real form} of $G$.
    \end{enumerate}
\end{definition}

As the name suggests, for a compact real form $K$ of $G$, we should expect its Lie algebra $\Lie(K)$ to be a compact real form of $\Lie(G)$.

\begin{lemma}\label{lem_realform}
    For a complex reductive group $G$, a real form $K$ of $G$ exists, for which its Lie algebra $\Lie(K)$ is a compact real form of $\Lie(G)$.
\end{lemma}

\begin{proof}
    We first assume that $G$ is semisimple. By
    \cite[Proposition 26.4]{fulton-harris_representationtheory} and its following discussion, we know that there exists a unique real form $\k$ of $\Lie(G)$, up to inner automorphism, that induces a compact real Lie subgroup $K$ of $G$ (through the exponential map). To show that $K$ is also a compact real form, we note that $\k$ yields a decomposition $\Lie(G)\cong\k\oplus\underline{p}$ at the level of real Lie algebras, induced by a Cartan involution $\theta$, due to \cite[VI.2. Corollary 6.22]{knapp_liegroups}. From this, we can use \cite[Theorem 6.31 (g)]{knapp_liegroups} to follow that $K$ is a maximal compact subgroup of $G$, i.e., a compact real form. In the general case where $G$ reductive, we know that $[G,G]$ is semisimple, for which the above argument yields a compact real form $K'$ of $[G,G]$, whose Lie algebra $\Lie(K')$ is a compact real form of $\Lie([G,G])$. Since $G=[G,G]\Rad(G)$ due to \cite[IV. 14.2 Proposition]{borel_linearalgebraicgroups}, and $\Rad(G)\cong(\C^\times)^q$ is a torus due to \cite[IV. 11.21 Proposition]{borel_linearalgebraicgroups}, we choose a compact real Lie subgroup $R\cong\U(1)^q$ of $\Rad(G)$, and define $K=K'R$. It follows that $K$ is a compact real form of $G$, where $\Lie(K)$ is a compact real form of $\Lie(G)$.
\end{proof}

\begin{example}\label{ex_realform}
    \begin{enumerate}[label=(\alph*)]
        \item\label{ex_realforma} The real Lie group of unitary matrices $\U(r)$ has the Lie algebra $\Lie(\U(r))$ of skew-Hermitian matrices, which has a negative-semidefinite Killing-form. Since the complexification of $\Lie(\U(r))$ is $\Lie(\GL(r,\C))$. We follow that $\U(r)$ is a real form of $\GL(r,\C)$, using Lemma \ref{lem_realform}.
              A similar argument shows that $\SU(r)$ is a real form of $\SL(r,\C)$.

        \item\label{ex_realformb} $\Sp(2n,\C)$ has a real form $K:=\Sp(2n,\C)\cap\U(r)$, with the Lie algebra $\Lie(K)=\Lie(\Sp(2n,\C))\cap\Lie(\U(r))$.

        \item\label{ex_realformc} $\SO(r,\C)$ has a real form $K:=\SO(r,\C)\cap\U(r)$, with the Lie algebra $\Lie(K)=\Lie(\SO(r,\C))\cap\Lie(\U(r))$.
    \end{enumerate}
\end{example}

Real forms are constructed in a concrete way, for most complex reductive groups of interest, by using a polar decomposition, as seen in this remark.

\begin{remark}
    Assuming $G$ is a subgroup of $\GL(r,\C)$ closed under conjugate-transposition, there exists a complex vector space of skew-Hermitian matrices $\u$, such that $\Lie(G)$ is the complexification of $\u$ as a complex vector space. We can decompose $\u$ as $\u=\a+\i\b$, where $\a$ is a set of skew-symmetric matrices, and $\b$ is a set of symmetric matrices. We then get $\k=\a+\b$ as a compact real form of $\Lie(G)$, such that $\Lie(G)=\k+\i\k$.
\end{remark}

Let us fix a Cartan subgroup $T\cong(\C^\times)^r$ of $G$, and a real form $K$ of $G$. For $H=K\cap T$, we have that $H\cong\U(1)^r$ is a maximal compact torus of $K$.
We want algebraic characterizations of the fundamental groups of $T$, $G$, $[G,G]$, and $G_{\ab}$, as $\mathbb{Z}$-modules.
Some of these fundamental groups embed naturally as lattices of the real vector spaces, as defined in \cite[1]{friedman-morgan_ontheconversetoatheoremofatiyahbott}:
\begin{align*}
    \Lie(G)_\R:=\i\Lie(K), &  & \Lie(T)_\R:=\Lie(T)\cap\Lie(G)_\R, &  & \zbf(\Lie(G))_\R:=\zbf(\Lie(G))\cap\Lie(G)_\R,
\end{align*}
such that $\Lie(H)=\Lie(K)\cap\Lie(T)=\i\Lie(G)_{\R}\cap\Lie(T)=\i\Lie(T)_\R$.

\begin{remark}\label{rem_lattice}
    Due to $T\cong(\C^\times)^r$ and $H\cong\U(1)^r$, we have isomorphisms:
    \begin{equation*}
        \pi_1(T)\cong\pi_1(H)\cong\Hom(\U(1),H)\cong\mathbb{Z}^r,
    \end{equation*}
    where $\Hom(\U(1),H)$ are the \textit{cocharacters} of $H$.
    Since $\U(1)$ and $H$ are connected, we have an injection from $\Hom(\U(1),H)$ to the derivatives:
    \begin{equation*}
        \Hom(\Lie(\U(1)),\Lie(H))=\Hom_\R(\i\R,\i\Lie(T)_\R)\cong\Hom_\R(\R,\Lie(T)_\R)\cong \Lie(T)_{\R},
    \end{equation*}
    where the last isomorphism is given by $[2\pi\mapsto X]\mapsto X$.
    Thus, $\pi_1(T)$ is isomorphic to the full-rank lattice $\Gamma:=\{X\in \Lie(T)_\R\ \vert\ \exp_H(2\pi i X)=e\}$ of $\Lie(T)_{\R}\cong\R^r$, called the \textit{kernel lattice}.
\end{remark}

We now recall some constructions from abstract root systems, applied to the root system $\Phi(G,T)$ of $G$ on $(
    \mathbf{E},\langle\_,\_\rangle)$, the Euclidean space isomorphic to $\R^r$ where $\langle\_,\_\rangle$ is a nondegenerate symmetric form induced from the Killing-form in \cite[II.8.5]{humphreys_introductiontoliealgebras}.

\begin{definition}\label{def_coroot}
    For a root $\alpha\in\Phi(G,T)$:
    \begin{enumerate}[label=(\alph*)]
        \item The \textit{coroot} of $\alpha$ is $\alpha^\vee=2\alpha/\langle\alpha,\alpha\rangle$. The set of coroots is denoted by $\Phi^\vee(G,T)$.

        \item The \textit{dual-root} of $\alpha$ is $H_\alpha\in\mathbf{E}^*$, defined  by through isomorphism $\mathbf{E}\cong \mathbf{E}^*$ induced by the Killing-form. The set of dual-roots is denoted by $\Phi_H(G,T)$.

        \item The \textit{dual-coroot} of $\alpha$ is $H_\alpha^\vee\in\mathbf{E}^*$. The set of dual-coroots is denoted by $\Phi^\vee_H(G,T)$.
    \end{enumerate}
\end{definition}

By \cite[III 9.2]{humphreys_introductiontoliealgebras}, it is known that $\Phi^\vee(G,T)$ is a root system of $(\mathbf{E},\langle\_,\_\rangle)$. It is also clear that dual-roots and dual-coroots form root systems.

\begin{remark}\label{rem_lattice2}
    Note that $\mathbf{E}$ is a subspace of $\Lie(T)$. Furthermore, the \textit{dual-coroot lattice} $\Lambda:=\Span_\mathbb{Z}(\Phi_H^\vee(G,T))$ is contained within the kernel lattice $\Gamma$ from Remark \ref{rem_lattice}, as proven in \cite[Lemma 12.8]{hall_liegroupsliealgebrasandreps}.
\end{remark}

We can now search for expressions of $\pi_1(G)$, $\pi_1([G,G])$, and $\pi_1(G_{\ab})$, following \cite[$\S$2]{friedman-morgan_ontheconversetoatheoremofatiyahbott} and \cite[$\S$3.3]{ho-liu_yangmillsconnectionsorientable}.

\begin{theorem}\label{thm_fundgr}
    For a complex reductive group $G$ and a Cartan subgroup $T$:
    \begin{enumerate}[label=(\roman*)]
        \item The inclusion $T\hookrightarrow G$ induces a surjection $\pi_1(T)\rightarrow\pi_1(G)$, such that $\pi_1(G)\cong\Gamma/\Lambda$.
    \end{enumerate}
    Let $\widehat{\Lambda}$ be the saturation of $\Lambda$ in $\Gamma$, i.e., the sublattice $\widehat{\Lambda}$ of $\Gamma$ containing $\Lambda$, minimal with respect to inclusion, such that $\Gamma/\widehat{\Lambda}$ is free.
    We claim that:
    \begin{enumerate}[label=(\roman*),resume]
        \item The group $\pi_1(G_{\ab})\cong\mathbb{Z}^q$ is a lattice, and the group $\pi_1([G,G])$ is finite.

        \item The short exact sequence $1\rightarrow [G,G]\rightarrow G\rightarrow G_{\ab}\rightarrow1$ induces the short exact sequence:
              \begin{equation}\label{eq_fundgr1}
                  1\rightarrow\pi_1([G,G])\rightarrow\pi_1(G)\rightarrow\pi_1(G_{\ab})\rightarrow 1,
              \end{equation}
              of fundamental groups.
              In particular, $\pi_1(G_{\ab})\cong\Gamma/\widehat{\Lambda}$ and $\pi_1([G,G])\cong\widehat{\Lambda}/\Lambda$.
    \end{enumerate}
\end{theorem}

\begin{proof}
    For the proof of \textit{(i)}, see \cite[Corollary 13.18]{hall_liegroupsliealgebrasandreps}, noting that $T\hookrightarrow G$ is the complexified version of $H\hookrightarrow K$.
    In particular, the surjectivity of $\pi_1(T)\rightarrow\pi_1(G)$ is proven in \cite[Proposition 13.37]{hall_liegroupsliealgebrasandreps}.

    We prove \textit{(ii)}. Since $G_{\ab}\cong(\C^\times)^q$, we have that $\pi_1(G_{\ab})\cong\mathbb{Z}^q$ is a lattice.
    As $[G,G]$ has a semisimple real form $[K,K]$, Lie algebra cohomology implies that $\pi_1([G,G])$ is finite, as seen in \cite[Theorem 16.1]{chevalley-eilenberg_cohomologytheoryofliegroupsliealgebras}.

    We now prove \textit{(iii)}. Since $G\rightarrow G_{\ab}$ is a fibration, with fibers isomorphic to $[G,G]$, it induces the long exact homotopy sequence:
    \begin{equation*}
        \ldots\rightarrow\pi_2(G_{\ab})\rightarrow\pi_1([G,G])\rightarrow\pi_1(G)\rightarrow\pi_1(G_{\ab})\rightarrow\pi_0([G,G])\rightarrow\ldots
    \end{equation*}
    Since $[G,G]$ is connected, we have $\pi_0([G,G])\cong1$. Using Morse theory results from \cite{bott_morsetheorytopologyofliegroups}, it can be proven that $\pi_2(G_{\ab})\cong1$, as $G_{\ab}$ is a Lie group. Therefore, the sequence in (\ref{eq_fundgr1}) is exact.
    The isomorphisms $\pi_1(G_{\ab})\cong\Gamma/\widehat{\Lambda}$, and $\pi_1([G,G])\cong\widehat{\Lambda}/\Lambda$, follow directly from (\ref{eq_fundgr1}) and \textit{(i)}, \textit{(ii)}.
\end{proof}

Harder-Narasimhan types and obstruction classes will be defined through an element in $\pi_1(G_{\ab})$, so we need to embed $\pi_1(G_{\ab})$ into a lattice inside $\Lie(T)_\R$.

\begin{remark}\label{rem_lattice3}
    \begin{enumerate}[label=(\alph*)]
        \item\label{rem_lattice3a} We denote the character lattices of $G$ and $G_{\ab}$ by $\Xbf(G)$ and $\Xbf(G_{\ab})$ respectively. For a real form $K_{\ab}$ of $G_{\ab}$, we have:
              \sloppy
              \begin{equation*}
                  \Xbf(G)\cong \Xbf(G_{\ab})\cong\Hom(K_{\ab},\U(1)).
              \end{equation*}
              The first isomorphism uses that characters $\chi:G\rightarrow\C^\times$ map to an abelian group $\C^\times$.
              Since $K_{ab}$ and $\U(1)$ are connected, we have an injection from $\Hom(K_{\ab},\U(1))$ to the derivatives:
              \begin{align*}
                  \Hom(\zbf(\Lie(K)),\Lie(\U(1))) & =\Hom_{\R}(\i\zbf(\Lie(G))_\R,\i\R)                       \\
                                                  & \cong\Hom_{\R}(\zbf(\Lie(G))_\R,\R)=\zbf(\Lie(G))^*_{\R}.
              \end{align*}
              Thus, $\Xbf(G)$ is isomorphic to the full-rank lattice:
              \begin{equation*}
                  \Theta^*:=\left\{f\in\zbf(\Lie(G))^*_{\R}\ \middle\vert\ f(\i\ker(\exp_{K_{\ab}}))\subseteq\Span_\mathbb{Z}(2\pi)\right\},
              \end{equation*}
              of $\zbf(\Lie(G))^*_{\R}\cong\R^q$.

        \item\label{rem_lattice3b} The isomorphisms $T\cong (\C^\times)^r$ and $G_{\ab}\cong (\C^\times)^q$ induce isomorphisms $\Xbf(G)\cong\pi_1(G_{\ab})$ and $\zbf(\Lie(G))_\R^*\cong\zbf(\Lie(G))_\R$. Thus, $\pi_1(G_{\ab})$ embeds into the lattice $\Theta^*$ of $\zbf(\Lie(G))_\R^*$, isomorphic to a lattice $\Theta$ of $\zbf(\Lie(G))_\R$.

              Through the projection $T\rightarrow G_{\ab}$, we can induce a map $\Lie(T)_\R\rightarrow\zbf(\Lie(G))_\R$, and through the isomorphism $G_{\ab}\cong\Rad(G)$, we can induce a map $\zbf(\Lie(G))_\R\rightarrow\Lie(T)_\R$. We compose these maps as $\varphi:\zbf(\Lie(G))_{\R}\rightarrow\zbf(\Lie(G))_{\R}$, which is an isomorphism.
              From this, we induce an embedding of $\pi_1(G_{\ab})$ into $\zbf(\Lie(G))_{\R}$, as $\Psi=\varphi^{-1}(\Theta)$.
    \end{enumerate}
\end{remark}

This extra step in \ref{rem_lattice3b} may seem strange at first, as we had already found a way to embed $\pi_1(G_{\ab})$ into $\zbf(\Lie(G))_{\R}$. Modifying the lattice through $\Psi=\varphi^{-1}(\Theta)$ is akin to mapping from degrees to the slopes in the vector bundle case, as we will see in Example \ref{ex_glrClattice} and Subsection \ref{subsec_hntypes}. This is required to verify the slope conditions of Harder-Narasimhan filtrations. Lastly, we need to encode the slope conditions of canonical reductions within $\Lie(T)_\R$. For this, we use Weyl chambers. Let us fix a Borel subgroup $B$ of $G$, containing $T$, inducing simple roots $\triangle$.

\begin{definition}\label{def_weylch}
    The set $C_B:=\{X\in\Lie(T)_\R\ \vert\ \forall\alpha\in\triangle:\alpha(X)\geq0\}$ is the \textit{(closed) Weyl chamber} of $B$.
\end{definition}

In the Biswas and Holla approach to canonical reductions, from Subsection \ref{subsec_bh}, we made use of extensions to Levi-factors, so we will also need to apply Theorem \ref{thm_fundgr} and Remark \ref{rem_lattice3} to Levi-factors.

\begin{remark}\label{rem_latticelevi}
    For a standard parabolic subgroup $P_I\cong U_I L_I$ of $G$, we have that $T$ is a Cartan subgroup of the Levi-factor $L_I$, which is a complex reductive group, and that $L_I\cap B$ is a Borel subgroup of $L_I$.
    Using Theorem \ref{thm_fundgr}, we have a dual-coroot lattice $\Lambda_I=\Span_{\mathbb{Z}}(\Phi_H^\vee(L_I,T))$ of $\Lie(T)_\R\cong\R^r$, such that $\pi_1(L_I)\cong\Gamma/\Lambda_I$. We obtain the short exact sequence:
    \begin{equation*}
        1\rightarrow\pi_1([L_I,L_I])\rightarrow\pi_1(L_I)\rightarrow\pi_1((L_I)_{\ab})\rightarrow 1.
    \end{equation*}
    Since $\Rad(G)\subseteq\Rad(L_I)$, we have $G_{\ab}\cong(\C^\times)^q$ and $(L_I)_{\ab}\cong(\C^\times)^{q'}$, for $0\leq q\leq q'\leq r$, such that $\pi_1((L_I)_{\ab})$ embeds into $\zbf(L_I)_{\R}\cong\R^{m'}$ as a full-rank lattice $\Psi_I$.
\end{remark}

These abstract constructions will be made much clearer through examples, where the real forms, fundamental groups, and lattices, are made explicit.

\begin{example}\label{ex_glrClattice}
    For $G=\GL(r,\C)$, we use our standard choices of Cartan and Borel subgroups, from Subsection \ref{subsec_slopesemistab}, with the isomorphism $T\cong(\C^\times)^r$, mapping every diagonal entry of a matrix to the corresponding component in $(\C^\times)^r$. Thus, we have $\pi_1(T)\cong\mathbb{Z}^r$.

    Let $X^H$ denote the conjugate transpose of a matrix $X$. With the real form from Example \ref{ex_realform} \ref{ex_realforma}, we calculate the following real vector spaces, and the Weyl chamber:
    \begin{align*}
        \Lie(\GL(r,\C))_\R       & =\i\Lie(\U(r))=\left\{X\in\Lie(\GL(r,\C))\ \middle\vert\ X-X^H=0\right\},                                                           \\
        \Lie(T)_\R               & =\Lie(T)\cap\Lie(\GL(r,\C))_\R=\left\{\begin{pmatrix}
                                                                             z_1 & 0      & 0   \\
                                                                             0   & \ddots & 0   \\
                                                                             0   & 0      & z_r
                                                                         \end{pmatrix}\ \middle\vert\ z_1,\ldots,z_r\in\R  \right\},                                   \\
        \zbf(\Lie(\GL(r,\C)))_\R & =\zbf(\Lie(\GL(r,\C)))\cap\Lie(\GL(r,\C))_\R=\Span_\R(I_r),                                                                         \\
        C_B                      & =\left\{\begin{pmatrix}
                                               z_1 & 0      & 0   \\
                                               0   & \ddots & 0   \\
                                               0   & 0      & z_r
                                           \end{pmatrix}\in\Lie(T)_\R\ \middle\vert\ z_1\geq\ldots\geq z_{r}\right\}.
        \intertext{Furthermore, we use the simple roots $\triangle_H^\vee$ of the dual-coroots $\Phi_H^\vee(\GL(r,\C),T)$ to find the following lattices:}
        \Gamma                   & =\left\{\begin{pmatrix}
                                               z_1 & 0      & 0   \\
                                               0   & \ddots & 0   \\
                                               0   & 0      & z_r
                                           \end{pmatrix}\ \middle\vert\ z_1,\ldots,z_r\in\mathbb{Z}  \right\},                                                         \\
        \Lambda                  & =\left\{\begin{pmatrix}
                                               z_1 & 0      & 0   \\
                                               0   & \ddots & 0   \\
                                               0   & 0      & z_r
                                           \end{pmatrix}\in\Gamma\ \middle\vert\ \sum_{i=1}^nz_i=0\right\},
        \qquad\widehat{\Lambda}     = \Lambda,                                                                                                                         \\
        \Theta^*                 & = \left\{f\in\zbf(\Lie(\GL(r,\C)))_{\R}^*\ \middle\vert\ f(\Span_{\mathbb{Z}}(2\pi I_r))\subseteq \Span_{\mathbb{Z}}(2\pi)\right\}, \\
        \Theta                   & = \Span_\mathbb{Z}(I_r),                                                                                                            \\
        \Psi                     & = (1/r)\Span_\mathbb{Z}(I_r).
    \end{align*}
    From this, we conclude that $\pi_1(\SL(r,\C))\cong 1$, $\pi_1(\GL(r,\C))\cong \mathbb{Z}$, and $\pi_1(\GL(r,\C)_{\ab})\cong \mathbb{Z}$, using Theorem \ref{thm_fundgr}.

    Following Remark \ref{rem_latticelevi}, we also explain how these constructions appear for Levi-factors. We restrict to maximal standard parabolic subgroups $P_I= U_I L_I$ of $\GL(r,\C)$, where $I=\{\alpha_{l,l+1}\}\subseteq\triangle$.
    We calculate the following real vector spaces and lattices:
    \begin{align*}
        \zbf(\Lie(L_I))_\R & =\zbf(\Lie(L_I))\cap\Lie(L_I)_\R=\left\{\begin{pmatrix}
                                                                         z_1I_l & 0          \\
                                                                         0      & z_2I_{r-l}
                                                                     \end{pmatrix}\ \middle\vert\ z_1,z_2\in\R \right\},                                                         \\
        \Lambda_I          & =\left\{\begin{pmatrix}
                                         z_1 & 0      & 0   \\
                                         0   & \ddots & 0   \\
                                         0   & 0      & z_r
                                     \end{pmatrix}\in\Gamma\ \middle\vert\ \sum_{i=1}^lz_i=\sum_{i=l+1}^rz_i=0\right\},\qquad\widehat{\Lambda}_I = \Lambda_I,                    \\
        \Theta_I^*         & = \left\{f\in\zbf(\Lie(L_I))_{\R}^*\ \middle\vert\ z_1,z_2\in\mathbb{Z}: f\left(2\pi\begin{pmatrix}
                                                                                                                         z_1I_l & 0          \\
                                                                                                                         0      & z_2I_{r-l}
                                                                                                                     \end{pmatrix}\right)\subseteq \Span_{\mathbb{Z}}(2\pi)\right\}, \\
        \Theta_I           & = \left\{\begin{pmatrix}
                                          z_1I_l & 0          \\
                                          0      & z_2I_{r-l}
                                      \end{pmatrix}\ \middle\vert\ z_1,z_2\in\mathbb{Z}  \right\},                                                                               \\
        \Psi_I             & = \left\{\begin{pmatrix}
                                          (z_1/l)I_l & 0                  \\
                                          0          & (z_2/(r-l))I_{r-l}
                                      \end{pmatrix}\ \middle\vert\ z_1,z_2\in\mathbb{Z}  \right\}.
    \end{align*}
    From this, we conclude that $\pi_1([L_I,L_I])\cong 1$, $\pi_1(L_I)\cong \mathbb{Z}^2$ and $\pi_1((L_I)_{\ab})\cong \mathbb{Z}^2$, using Theorem \ref{thm_fundgr}. These results can be generalized to all standard parabolics $P_I$ and their Levi-factors $L_I$, which containing have more than two diagonal blocks.
\end{example}
The last lattice $\Psi_I$ of Example \ref{ex_glrClattice} is a good indicator of where we will record Harder-Narasimhan types, $z_1/l$ and $z_2/(r-l)$ appear like slopes of quotients of a Harder-Narasimhan filtration. To improve our understanding, let us also include the case of $\Sp(2n,\C)$.

\begin{example}\label{ex_sprClattice}
    For $G=\Sp(2n,\C)$, we use our standard choices of Cartan and Borel subgroups, and all notation, from Subsection \ref{subsec_sympvecbun}, with the standard isomorphism $T\cong(\C^\times)^n$ via the mapping $T_{t_1,\ldots,t_n}=\Diag(t_1,\ldots,t_n,t_n^{-1},\ldots,t_1^{-1})\mapsto(t_1,\ldots,t_n)$. Thus, we have $\pi_1(T)\cong\Z^n$.

    Let us denote $Z_{z_1,\ldots,z_n}=\Diag(z_1,\ldots,z_n,-z_1,\ldots,-z_n)$ in $\Lie(T)$. With the real form from Example \ref{ex_realform} \ref{ex_realformb}, we calculate the following real vector spaces, and the Weyl chamber:
    \begin{align*}
        \Lie(\Sp(2n,\C))_\R       & =\i\Lie(K)=\left\{X\in \Lie(\GL(r,\C))\ \middle\vert\ M_{2n}X+X^TM_{2n}=0,X-X^H=0\right\},              \\
        \Lie(T)_\R                & =\Lie(T)\cap\Lie(\Sp(2n,\C))_\R=\left\{Z_{z_1,\ldots z_n}\ \middle\vert\ z_1,\ldots,z_n\in\R  \right\}, \\
        \zbf(\Lie(\Sp(2n,\C)))_\R & = 0,                                                                                                    \\
        C_B                       & =\left\{Z_{z_1,\ldots z_n}\in \Lie(T)_{\R}\ \middle\vert\ z_1\geq\ldots\geq z_n\geq 0\right\}.
        \intertext{Furthermore, we use the simple roots $\triangle_H^\vee$ of the dual-coroots $\Phi_H^\vee(\Sp(2n,\C),T)$ to find the following lattices:}
        \Gamma                    & =\left\{Z_{z_1,\ldots z_n}\ \middle\vert\ z_1,\ldots,z_n\in\mathbb{Z}  \right\},                        \\
        \Lambda                   & =\Gamma,\qquad\widehat{\Lambda}=\Gamma,                                                                 \\
        \Theta^*                  & =0, \qquad \Theta=0, \qquad \Psi=0.
    \end{align*}

    From this, we conclude that $\pi_1([\Sp(2n,\C),\Sp(2n,\C)])\cong1$, $\pi_1(\Sp(2n,\C))\cong1$ and $\pi_1(\Sp(2n,\C)_{\ab})\cong 1$, using Theorem \ref{thm_fundgr}. Following Remark \ref{rem_latticelevi}, we also explain how these constructions appear for Levi-factors. We restrict to maximal standard parabolic subgroups $P_I= U_I L_I$ of $\Sp(2n,\C)$, where $I=\{\alpha_{l,l+1}\},\{2\alpha_{l}\}\subseteq\triangle$.
    We calculate the following real vector spaces and lattices:
    \begin{align*}
        \zbf(\Lie(L_I))_\R & =\zbf(\Lie(L_I))\cap\Lie(L_I)_\R=\left\{\begin{pmatrix}
                                                                         zI_l & 0 & 0     \\
                                                                         0    & 0 & 0     \\
                                                                         0    & 0 & -zI_l \\
                                                                     \end{pmatrix}\ \middle\vert\ z\in\R  \right\},                                                        \\
        \Lambda_I          & =\left\{Z_{z_1,\ldots z_n}\in \Gamma\ \middle\vert\ \sum_{i=1}^lz_i=0\right\},  \qquad \widehat{\Lambda}_I = \Lambda_I,                       \\
        \Theta_I^*         & = \left\{f\in\zbf(\Lie(L_I))_{\R}^*\ \middle\vert\ z\in\mathbb{Z}: f\left(2\pi\begin{pmatrix}
                                                                                                                   zI_l & 0 & 0     \\
                                                                                                                   0    & 0 & 0     \\
                                                                                                                   0    & 0 & -zI_l \\
                                                                                                               \end{pmatrix}\right)\subseteq \Span_{\mathbb{Z}}(2\pi)\right\}, \\
        \Theta_I           & = \left\{\begin{pmatrix}
                                          zI_l & 0 & 0     \\
                                          0    & 0 & 0     \\
                                          0    & 0 & -zI_l \\
                                      \end{pmatrix}\ \middle\vert\ z\in\mathbb{Z}  \right\},                                                                               \\
        \Psi_I             & = \left\{\begin{pmatrix}
                                          (z/l)I_l & 0 & 0         \\
                                          0        & 0 & 0         \\
                                          0        & 0 & (-z/l)I_l \\
                                      \end{pmatrix}\ \middle\vert\ z\in\mathbb{Z}  \right\}.
    \end{align*}
    From this, we conclude that $\pi_1([L_I,L_I])\cong1$, $\pi_1(L_I)\cong\mathbb{Z}$ and $\pi_1(L_I)_{\ab}\cong\mathbb{Z}$, using Theorem \ref{thm_fundgr}.

    These results can definitely be generalized for all standard parabolics of $\Sp(2n,\C)$, and Levi-factors, which may have more diagonal blocks.
\end{example}

\begin{example}
    For $G=\SO(r,\C)$, we will not present the details, but wish to remark that unlike the cases of $\GL(r,\C)$ and $\Sp(2n,\C)$, there is a difference between the dual-coroot lattices $\Lambda$ and $\widehat{\Lambda}$, since $\Gamma/\Lambda$ has torsion elements. It turns out that $\pi_1([\SO(r,\C),\SO(r,\C)]) = \pi_1(\SO(r,\C))\cong \mathbb{Z}/2\mathbb{Z}$, and $\pi_1(\SO(r,\C)_{\ab})\cong 1$, using Theorem \ref{thm_fundgr}.
\end{example}

\subsection{Obstruction classes and Harder-Narasimhan types}\label{subsec_hntypes}

Now that we understand how fundamental groups of reductive groups appear as lattices, we can use them to define obstruction classes of principal bundles, which we can use to define Harder-Narasimhan types. For the universal covering $\widetilde{G}$ of $G$, which is a simply connected complex Lie group, we have the fibration:
\begin{equation}\label{eq_fibration}
    1\rightarrow\pi_1(G)\rightarrow\widetilde{G}\overset{p}{\rightarrow} G\rightarrow1,
\end{equation}
such that $p$ is a fiber bundle, and a central extension of groups. In particular, $\pi_1(G)$ is abelian.

We denote the structure sheaf of $X$ by $\O_X$.
By passing (\ref{eq_fibration}) through Čech cohomology, we obtain the maps between sets:
\begin{equation*}
    \ldots\rightarrow \check{H}^1(X,\O_X(\widetilde{G}))\overset{p_1}{\rightarrow} \check{H}^1(X,\O_X(G))\overset{o_2}{\rightarrow}\check{H}^2(X,\O_X(\pi_1(G))).
\end{equation*}
By construction, the sets $\check{H}^1(X,\O_X(\_))$ correspond to sets of isomorphism classes of principal bundles. These isomorphism classes can be represented by cocycles of principal bundles, which correspond to $1$-cocycles in the sense of Čech cohomology. As $X$ is a compact Riemann surface, the singular cohomologies $H^*(X,\pi_1(G))$ and Čech cohomologies $\check{H}^*(X,\O_X(\pi_1(G)))$ are naturally isomorphic, as seen in \cite[Chapter 4. 4.14. Theorem, Chapter 4. 5.3. Theorem]{ramanan_globalcalculus}. Applying Poincaré duality then gives us a canonical isomorphism $\check{H}^2(X,\O_X(\pi_1(G)))\cong \pi_1(G)$. We now investigate the connecting map $o_2:\check{H}^1(X,\O_X(G))\rightarrow\check{H}^2(X,\O_X(\pi_1(G)))$. For a principal $G$-bundle $\xi$, its isomorphism class $[\xi]\in\check{H}^1(X,\O_X(G))$ is represented by the cocycles $(\sigma_{i,j})_{i,j\in J}$.
For all $i,j\in J$, we use the universal lifting property of covering spaces to lift $\sigma_{i,j}$ to $\widetilde{\sigma_{i,j}}$, such that the following diagram commutes:
\begin{equation*}
    \begin{tikzpicture}[scale=1.3]
        \node (A) at (0,0) {$U_i\cap U_j$};
        \node (B) at (3,0.8) {$\widetilde{G}$};
        \node(C) at (3,0) {$G$};
        \path[->,font=\scriptsize,>=angle 90]
        (A) edge node[above]{$\widetilde{\sigma_{i,j}}$} (B)
        (A) edge node[above]{$\sigma_{i,j}$} (C)
        (B) edge node[right]{$p$} (C);
    \end{tikzpicture}
\end{equation*}
In general, these lifts are not unique, as $\sigma_{i,j}$ can have two different lifts that vary within the fibers of $p:\widetilde{G}\rightarrow G$.
Furthermore, we cannot guarantee that the cocycle conditions for $(\widetilde{\sigma_{i,j}})_{i,j\in I}$, are fulfilled.
We can measure the defect of this failure by defining for all $i,j,k\in J$:
\begin{equation*}\label{eq_2cocycle}
    \sigma_{i,j,k}=\widetilde{\sigma_{j,k}}\widetilde{\sigma_{i,k}}^{-1}\widetilde{\sigma_{i,j}}:U_i\cap U_j\cap U_k\rightarrow\pi_1(G).
\end{equation*}
This defines a $2$-cocycle $(\sigma_{i,j,k})_{i,j,k\in J}$ representing $o_2([\xi])\in\check{H}^2(X,\O_X(\pi_1(G)))$.
If $o_2([\xi])=1\in\check{H}^2(X,\O_X(\pi_1(G)))$, then $[\xi]$ is equivalently the image of an element $[\widetilde{\xi}]\in\check{H}^1(X,\O_X(\widetilde{G}))$, i.e., there exists a principal $\widetilde{G}$-bundle $\widetilde{\xi}$ such that $\widetilde{\xi}(G)\cong\xi$.

\begin{definition}\label{def_obstclass}
    Let $\xi$ be a principal $G$-bundle on $X$. Via the isomorphism $\check{H}^2(X,\O_X(\pi_1(G)))\cong\pi_1(G)$, the class $o_2([\xi])\in\check{H}^2(X,\O_X(\pi_1(G)))$ induces an element $o_2(\xi)\in\pi_1(G)$, called the \textit{(second) obstruction class} of $\xi$.
\end{definition}

Through the surjection $\pi_1(G)\rightarrow\pi_1(G_{\ab})$ from Theorem \ref{thm_fundgr}, we can further map the obstruction class $o_2(\xi)$ to an element $\widetilde{o_2}(\xi)\in\pi_1(G_{\ab})$.
As before, we fix a Cartan subgroup $T$ of $G$, and an isomorphism $T\cong(\C^\times)^r$.
We can now embed the obstruction class of $\xi$ into $\Lie(T)_{\R}$ to define its topological type.

\begin{definition}\label{def_type}
    The fundamental group $\pi_1(G_{\ab})$ embeds as a lattice $\Psi$ of $\zbf(\Lie(G))_\R$, due to Remark \ref{rem_lattice3} \ref{rem_lattice3b}.
    Through this, $\widetilde{o_2}(\xi)$ defines an element $\mu^\xi\in\Lie(T)_{\R}$, called the \textit{(topological) type} of $\xi$.
\end{definition}

\begin{remark}\label{rem_obstclassvsdegree}
    \begin{enumerate}[label=(\alph*)]
        \item\label{rem_obstclassvsdegreea} For $G=\GL(r,\C)$, we have that the type $\mu^\xi$ is a matrix $(d/r)I_r$, with $d\in\mathbb{Z}$, due to Example \ref{ex_glrClattice}. In particular, we recover $\deg \xi(\C^r) = d$ as the trace of that matrix, so $\mu^\xi$ indeed determines the topological type of $\xi$.

        \item\label{rem_obstclassvsdegreeb} For a vector bundle $E$ of rank $r$, we have $\mu^{\Fr(E)}=\mu(E)I_r$.
              To prove this, we could use that $\deg(E)=\int_Xc_1(E)$, where $c_1(E)$ denotes the \textit{first Chern class} of $E$. We would need to show that the first Chern class corresponds to the second obstruction class.
    \end{enumerate}
\end{remark}

The following remark explains that obstruction classes and types are functorial.

\begin{remark}\label{rem_obstructionfunctorial}
    Let $\varphi:G\rightarrow G'$ be a homomorphism of complex reductive Lie groups, where $G$ and $G'$. Let $\xi(G')$ be the extension of $\xi$ to $G'$ through $\varphi$.
    \begin{enumerate}[label=(\alph*)]
        \item Due to the functoriality of Čech cohomology, the induced morphism $\varphi_*:\pi_1(G)\rightarrow\pi_1(G')$ maps $o_2(\xi)$ to $o_2(\xi(G'))$.

        \item Since the abelianization of groups is functorial, the morphism $\varphi$ induces a morphism $\varphi_{\ab}:G_{\ab}\rightarrow G'_{\ab}$, such that the following diagram commutes:
              \begin{equation*}
                  \begin{tikzpicture}[scale=1.3]
                      \node (A) at (0,1) {$\pi_1(G)$};
                      \node (C) at (0,0) {$\pi_1(G_{\ab})$};
                      \node (B) at (2,1) {$\pi_1(G')$};
                      \node(D) at (2,0) {$\pi_1(G'_{\ab})$};
                      \path[->,font=\scriptsize,>=angle 90]
                      (A) edge node[left]{} (C)
                      (B) edge node[right]{} (D)
                      (C) edge node[above]{$(\varphi_{\ab})_*$} (D)
                      (A) edge node[above]{$\varphi_*$} (B);
                  \end{tikzpicture}
              \end{equation*}
              Thus, $(\varphi_{\ab})_*:\pi_1(G_{\ab})\rightarrow\pi_1(G'_{\ab})$ maps $\widetilde{o_2}(\xi)$ to $\widetilde{o_2}(\xi(G'))$.

        \item For a Cartan subgroup $T\cong(\C^\times)^r$ of $G$, and a Cartan subgroup $T'\cong(\C^\times)^{r'}$ of $G'$, such that $\varphi(\Rad(G))\subseteq\Rad(G')$ and $\varphi(T)\subseteq T'$, the type is also functorial, i.e., $D\varphi$ maps $\mu^\xi$ to $\mu^{\xi(G')}$.
    \end{enumerate}
\end{remark}

We are now able to state the existence and uniqueness of canonical reductions, as proven in Subsections \ref{subsec_ab} and \ref{subsec_bh}, using types, as shown in \cite[Lemma 2.2.1]{friedman-morgan_ontheconversetoatheoremofatiyahbott}.

\begin{theorem}\label{thm_hntype}
    There exists a reduction $s^*\xi$ of $\xi$ to a standard parabolic subgroup $P_I= U_IL_I$ of $G$, unique up to conjugation in the sense of Lemma \ref{lem_adequivalentreduction}, such that:
    \begin{enumerate}[label=(\roman*)]
        \sloppy
        \item The extension of $s^*\xi$ to the Levi-factor $L_I$, denoted by $(s^*\xi)(L_I)$, is Ramanathan-semistable.

        \item For the type $\mu^{(s^*\xi)(L_I)}\in\Lie(T)_{\R}$, we have $\mu^{(s^*\xi)(L_I)}\in C_B$. Furthermore, for all $\alpha\in I$, we have $\alpha(\mu^{(s^*\xi)(L_I)})>0$.
    \end{enumerate}
    This reduction $s^*\xi$ is a canonical reduction of $\xi$.
\end{theorem}

\begin{proof}
    It suffices to show that the conditions \textit{(i)} and \textit{(ii)} are equivalent to the conditions \ref{bh1} and \ref{bh2} from Theorem \ref{thm_bh}.
    By definition, \textit{(i)} is equivalent to \ref{bh1} of Theorem \ref{thm_bh}.

    We prove that \textit{(ii)} is equivalent to \ref{bh2}, for which it suffices to show that \ref{bh2} implies \textit{(ii)}.
    Let $\chi:P_I\rightarrow\C^\times$ be any character, then:
    \begin{equation}\label{eq_hntype}
        D\chi(\mu^{s^*\xi})=\mu^{\chi(s^*\xi)}=\deg(\chi(s^*\xi)),
    \end{equation}
    where the first equality follows from the functoriality of types in Remark \ref{rem_obstructionfunctorial}, and the second follows from Remark \ref{rem_obstclassvsdegree} \ref{rem_obstclassvsdegreeb}.
    Similarly, for the restriction $\chi|_{L_I}:L_I\rightarrow\C^\times$, we have:
    \begin{equation}\label{eq_hntype2}
        D\chi(\mu^{(s^*\xi)(L_I)})=\mu^{\chi((s^*\xi)(L_I))}=\deg(\chi((s^*\xi)(L_I))).
    \end{equation}
    It follows that $D\chi(\mu^{(s^*\xi)(L_I)})=\deg(\chi(s^*\xi))$, since $\deg(\chi(s^*\xi))=\deg(\chi((s^*\xi)(L_I)))$. For all $\alpha\in\triangle\setminus I$, we have from decomposition of Levi-factors into weight spaces that all roots $\alpha\in\Span_\mathbb{Z}(\triangle\setminus I)$ act trivially on $\zbf(\Lie(L_I))$, hence $\alpha(\mu^{(s^*\xi)(L_I)})=0$, since
    $\mu^{(s^*\xi)(L_I)}\in\zbf(\Lie(L_I))_{\R}$.
    For $\alpha\in I$, let $\chi_\alpha:P_I\rightarrow\C^\times$ be a character as given in Lemma \ref{lem_canonreduct}. Using (\ref{eq_hntype}) and (\ref{eq_hntype2}), we have:
    \begin{equation*}
        D\chi_\alpha(\mu^{(s^*\xi)(L_I)})=\deg(\chi_\alpha(s^*\xi))>0,
    \end{equation*}
    which is a positive multiple of $\alpha(\mu^{(s^*\xi)(L_I)})$, so $\alpha(\mu^{(s^*\xi)(L_I)})>0$.
    Altogether, the type $\mu^{(s^*\xi)(L_I)}$ lies in $C_B$ and has the properties laid out in \textit{(ii)}.
\end{proof}

\begin{definition}\label{def_hntype}
    For a principal $G$-bundle $\xi$, and a reduction $s^*\xi$ of $\xi$ as in Theorem \ref{thm_hntype}, we call $\mu^{\xi}_{\HN}=\mu^{(s^*\xi)(L_I)}\in\Lie(T)_{\R}$ the \textit{Harder-Narasimhan type} of $\xi$.
\end{definition}

The Harder-Narasimhan type characterizes the topological type of any canonical reduction of $\xi$.
Due to the functoriality of types, the Harder-Narasimhan type $\mu^\xi_{\HN}$ also encodes the topological type $\mu^\xi$ of $\xi$ itself.
We will now see examples of the Harder-Narasimhan type $\mu^\xi_{\HN}$, and how it stores information about canonical reductions of $\xi$.

\begin{example}\label{ex_hntypehnfilt}
    Let $E$ be a vector bundle of rank $r$. We know that the canonical reduction $s^*\Fr(E)$ of $\Fr(E)$ to a standard parabolic subgroup $P_I$ of $\GL(r,\C)$ corresponds to the Harder-Narasimhan filtration of $E$:
    \begin{equation*}
        0=E_0\subsetneq\ldots\subsetneq E_t=E
    \end{equation*}
    with slope-semistable quotient bundles $F_{m}=E_{m}/E_{m-1}$, $m=1,\ldots,t$, of rank $r_m$, as seen in Theorem \ref{thm_hardernarasimhan}. By following Example \ref{ex_glrClattice}, and by noting that types coincide with slopes, as seen in Remark \ref{rem_obstclassvsdegree} \ref{rem_obstclassvsdegreeb}, the Harder-Narasimhan type $\mu^{\Fr(E)}_{\HN}$ of $\Fr(E)$ appears as:
    \begin{equation*}
        \mu^{\Fr(E)}_{\HN}=\begin{pmatrix}
            \mu(F_1)I_{r_1} & 0      & 0               \\
            0               & \ddots & 0               \\
            0               & 0      & \mu(F_t)I_{r_t} \\
        \end{pmatrix}\in\Lie(T)_{\R}.
    \end{equation*}
    From this, we see directly that the slope conditions of the Harder-Narasimhan filtration:
    \begin{equation*}
        \mu(F_1)>\ldots>\mu(F_t),
    \end{equation*}
    coincide with the inequalities given in \textit{(ii)} of Theorem \ref{thm_hntype}.
\end{example}

\begin{example}\label{ex_hntypespnfilt}
    Let $(E,\beta)$ be a symplectic vector bundle of rank $2n$, with the canonical reduction $s^*\Fr_{\Sp}(E,\beta)$ of $\Fr_{\Sp}(E,\beta)$ to a standard parabolic subgroup $P_I$ of $\Sp(2n,\C)$.
    We know that $s^*\Fr_{\Sp}(E,\beta)$ induces the symplectic Harder-Narasimhan filtration of $(E,\beta)$:
    \begin{equation*}
        0=E_0\subsetneq\ldots\subsetneq E_{t}\subsetneq E,
    \end{equation*}
    with semistable quotient bundles $F_{m}=E_{m}/E_{m-1}$, $m=1,\ldots,t$, of rank $r_m$, as seen in Theorem \ref{thm_hardernarasimhansymp}. Since the types of isotropic subbundles of $(E,\beta)$ coincide with slopes, similarly to Remark \ref{rem_obstclassvsdegree} \ref{rem_obstclassvsdegreeb}, the Harder-Narasimhan type $\mu^{\Fr_{\Sp}(E,\beta)}_{\HN}$ of $\Fr_{\Sp}(E,\beta)$ appears as:
    \begin{equation*}
        \mu^{\Fr_{\Sp}(E,\beta)}_{\HN}=\begin{pmatrix}
            \mu(F_1)I_{r_1} & 0      & \ldots          & \ldots & \ldots           & \ldots & 0                \\
            0               & \ddots & \ddots          & \ddots & \ddots           & \ddots & \vdots           \\
            \vdots          & \ddots & \mu(F_t)I_{r_t} & \ddots & \ddots           & \ddots & \vdots           \\
            \vdots          & \ddots & \ddots          & 0 I_s  & \ddots           & \ddots & \vdots           \\
            \vdots          & \ddots & \ddots          & \ddots & -\mu(F_t)I_{r_t} & \ddots & \vdots           \\
            \vdots          & \ddots & \ddots          & \ddots & \ddots           & \ddots & 0                \\
            0               & \ldots & \ldots          & \ldots & \ldots           & 0      & -\mu(F_1)I_{r_1} \\
        \end{pmatrix}\in\Lie(T)_{\R},
    \end{equation*}
    where the center block is $\mu^{\Fr(E_t^\perp/E_t,\beta)}_{\HN}=0 I_s$, where $s := \dim(E_t^\perp /E_t)$. From this, the slope conditions on the quotients of the symplectic Harder-Narasimhan filtration:
    \begin{equation*}
        \mu(F_{1})>\ldots>\mu(F_{t})>0,
    \end{equation*}
    coincide with the inequalities given in \textit{(ii)} of Theorem \ref{thm_hntype}.
\end{example}

As an application, we remark the use of Harder-Narasimhan types in stratifying moduli spaces of principal bundles.

\begin{remark}
    We fix a Cartan subgroup and Borel subgroup of $G$ such that $T\subseteq B\subseteq G$. The smooth moduli stack $Bun(G)$ of principal $G$-bundles on a Riemann surface $X$, where $G$ is a complex reductive group, has a stratification:
    \begin{equation*}
        Bun(G)=\bigsqcup_{\mu\in\Gamma\cap C_B} Bun_\mu(G),
    \end{equation*}
    into locally closed smooth algebraic substacks, where $\Gamma$ is the kernel lattice from Remark \ref{rem_lattice}, and $C_B$ is the Weyl chamber from Definition \ref{def_weylch}. The only open substack in this stratification is that of Ramanathan-semistable principal bundles $Bun^{\mathbf{sst}}(G)$.

    One can define a partial order $\geq$ on the pairs $(\mu,P_I)$ of Harder-Narasimhan types $\mu$, together with their corresponding standard parabolics $P_I$ in the canonical reductions, as done in \cite[Definition 2.4.1]{friedman-morgan_ontheconversetoatheoremofatiyahbott}:
    \begin{equation*}
        (\mu,P_I)\geq(\mu',P_{I'}')\text{ when }P_I\subseteq P_{I'}'\text{ and }\mu'\in\widehat{W\cdot\mu},
    \end{equation*}
    where $W\cdot\mu$ is the orbit of $\mu$ in $\Lie(T)_\R$ of the action of the Weyl group $W$, and $\widehat{W\cdot\mu}$ is the convex hull of $W\cdot\mu$. This order is compatible with the stratification such that:
    \begin{equation*}
        \overline{Bun_\mu(G)}=\bigsqcup_{(\mu,P_I)\geq(\mu',P_{I'}')} Bun_{\mu'}(G),
    \end{equation*}
    where $\overline{Bun_\mu(G)}$ is the closure of $Bun_\mu(G)$ in $Bun(G)$. This stratification can be seen as a colimit of GIT-instability (Hesselink) stratifications, as described more generally for the stack of coherent sheaves in \cite[4.4, 4.5, 4.6]{hoskins_stratificationsofmoduliofsheaves}.
\end{remark}

\end{document}